\documentclass[a4paper,11pt]{article}

\usepackage{a4wide}
\usepackage[utf8]{inputenc}
\usepackage[T1]{fontenc}
\usepackage{amsmath,amsthm,amsfonts,amssymb,bbm}
\usepackage{stackrel}
\usepackage{graphicx}
\usepackage{tikz}
\usepackage{url}
\usepackage{diagbox}

\newtheorem{theorem}{Theorem}[section]

\newtheorem{lemma}[theorem]{Lemma}

\newtheorem{assumption}[theorem]{Assumption}

\newcommand{\N}{\mathbb N}
\newcommand{\Z}{\mathbb Z}
\newcommand{\R}{\mathbb R}

\newcommand{\domain}{\mathrm{D}}

\newcommand{\cu}{\operatorname{curl}}
\newcommand{\di}{\operatorname{div}}

\newcommand{\trt}{\gamma_\mathrm{t}}

\DeclareMathOperator*{\essinf}{ess\,inf}

\newcommand{\vecA}{\vec{\mathcal A}}

\renewcommand{\vec}[1]{\underline{#1}}
\newcommand{\abs}[1]{\left\lvert{#1}\right\rvert}
\newcommand{\norm}[1]{{\left\lVert{#1}\right\rVert}}

\begin{document}

\title{Numerical study of conforming space-time methods for Maxwell's equations}
\author{Julia I.M.~Hauser$^1$, Marco~Zank$^2$}
\date{
	$^1$Institut für Wissenschaftliches Rechnen, \\
	Technische Universität Dresden, \\
	Zellescher Weg 25, 01217 Dresden, Germany \\[1mm]
	{\tt julia.hauser@tu-dresden.de}  \\[3mm]
	$^2$Fakult\"at f\"ur Mathematik, Universit\"at Wien, \\
	Oskar-Morgenstern-Platz 1, 1090 Wien, Austria \\[1mm]
	{\tt marco.zank@univie.ac.at}
}

\maketitle

\begin{abstract}
  Time-dependent Maxwell's equations govern electromagnetics. Under certain conditions, we can rewrite these equations into a partial differential equation of second order, which in this case is the vectorial wave equation. For the vectorial wave, we investigate the numerical application and the challenges in the implementation. For this purpose, we consider a space-time variational setting, i.e. time is just another spatial dimension. More specifically, we apply integration by parts in time as well as in space, leading to a space-time variational formulation with different trial and test spaces. Conforming discretizations of tensor-product type result in a Galerkin--Petrov finite element method that requires a CFL condition for stability. For this Galerkin--Petrov variational formulation, we study the CFL condition and its sharpness. To overcome the CFL condition, we use a Hilbert-type transformation that leads to a variational formulation with equal trial and test spaces. Conforming space-time discretizations result in a new Galerkin--Bubnov finite element method that is unconditionally stable. In numerical examples, we demonstrate the effectiveness of this Galerkin--Bubnov finite element method. Furthermore, we investigate different projections of the right-hand side and their influence on the convergence rates.
      
  This paper is the first step towards a more stable computation and a better understanding of vectorial wave equations in a conforming space-time approach.
\end{abstract}

\section{Introduction}
	
	The time-dependent Maxwell's equations and their discretization are required in many electromagnetic applications. One application is the modeling of an electric motor which is governed by the Ampère--Maxwell equation. The Ampère--Maxwell equation relates the time-dependent electric field with the current density and the magnetic field. If the corresponding space-time domain is star-shaped with respect to a ball, then we can rewrite Maxwell's system into the vectorial wave equation.
	Application of this technique in the static case is investigated in \cite{TIEGNA2013162}. Extensive studies of the static case can be found not only for analytic methods, but also for the approximation by finite elements, see e.g. \cite{ReviewElectricEdgeFEM}, or in a more applied static nonlinear example \cite{FESurrogateModel}. Moreover, the quasi-static case is examined, see \cite{3dimVectorPot}.
	
	However, the complete space-time setting of these equations has been studied in less detail. Most approaches for time-dependent Maxwell's equation are either dealing with the frequency domain, see \cite{Monk2003}, or using
	different types of	time-stepping methods, see \cite[Chapter~12.2]{Jin}. For the latter, proving stability is delicate. One solution is to stabilize the systems which stem from time-stepping methods. The development of such stabilized methods is an active research field, see e.g. \cite{XIE2021109896} on energy-preserving meshing for FDTD schemes or \cite{bai2022second} for the Cole--Cole model which involves polarization. Other classical time-stepping methods are the so-called leapfrog and the Newmark-beta methods, where the first is a special case of the second. Comparison and improvement of these methods can be found in \cite{crawford2020unconditionally}. Another class of time-stepping approaches is locally implicit time integrators, see e.g. \cite{HochbruckSturm2019} and references there.
	
	An alternative to time-stepping methods is a space-time approach. Most approaches apply discontinuous Galerkin methods to Maxwell's equations, see \cite{AbediMudaliar2017, dorfler2016space, EggerKretzschmarSchneppWeiland2015, LilienthalSchneppWeiland2014, XieWangZhang2013} and references there. Note that these methods are non-conforming in general.
	
	In this paper, we derive conforming space-time finite element discretizations for the vectorial wave equation without introducing additional unknowns, i.e. without a reformulation of the vectorial wave equation as a first-order system. The advantage of staying with the second-order formulation and of conforming methods is the need for a smaller number of degrees of freedom compared with first-order formulations or discontinuous Galerkin methods. First, we state the variational formulation for different trial and test spaces, i.e. a Galerkin--Petrov formulation. Second, using the so-called modified Hilbert transformation $\mathcal H_T$, introduced in \cite{SteinbachZankETNA2020, ZankDissBuch2020}, we present a variational formulation with equal trial and test spaces, i.e. a Galerkin--Bubnov formulation. The main goal of this paper is to derive and describe in detail these two conforming space-time methods, the assembling of the resulting linear systems and their comparison concerning stability and convergence. In more detail, we investigate the Galerkin--Petrov formulation and the corresponding CFL condition. Additionally, we elaborate on a Galerkin--Bubnov formulation using the modified Hilbert transformation, where all necessary mathematical tools are developed. In the end, we compare the results of the conforming Galerkin--Petrov and Galerkin--Bubnov finite element methods concerning their stability and convergence behavior by considering numerical examples. Moreover, we investigate the convergence behavior of both conforming space-time approaches when the right-hand side is approximated by two different projections.
	
	Before we introduce the space-time approach, we derive the vectorial wave equation from Maxwell's equations to get a better understanding of the properties of the equation. We state the main ideas but refer to \cite{HauserOhm2023,stern} and references for a more detailed derivation. To derive the vectorial wave equation, we need to rewrite Maxwell's equations using differential forms in a Lipschitz domain $Q\subset\R^4$. Hence, we define the Faraday 2-form $F := e \wedge  \mathrm{d}t + b,$	where $e$ is the $1$-form corresponding to the electric field $\vec E$ and $b$ the $2$-form corresponding to the magnetic flux density $\vec B$. Additionally, we define the Maxwell 2-form $G:= {h}\wedge  \mathrm{d}t - \tilde d$ and the source 3-form $\mathcal{J}:= \tilde{j} \wedge  \mathrm{d}t - \tilde{\rho},$ where $h$ is the $1$-form corresponding to the magnetic field $\vec H$, $\tilde d$ is the $2$-form corresponding to the electric flux density $\vec D$, $\tilde{j}$ is the $2$-form corresponding to the given electric current density $\vec j$ and $\tilde{\rho}$ is the $3$-form corresponding to the given charge density~$\rho$. Following the derivations in \cite{stern}, Maxwell's equations allow for the representation
	\begin{equation}  \label{Einf:eqn:Maxwell4D}
		\left.
		\begin{array}{rcl}
			\mathrm{d}\ F &=& 0,\\
			\mathrm{d}\ G &=& \mathcal{J},\\
			G &=& \star^{\epsilon,-\mu^{-1}} F,
		\end{array}
		\right \}
	\end{equation}
	where the four-dimensional exterior derivative $\mathrm{d}$ is deployed. Here, the Hodge star operator $\star^{\epsilon,-\mu^{-1}}$ is weighted by $\epsilon$ for the components of $(\mathrm{d}x^{01},\mathrm{d}x^{02},\mathrm{d}x^{03})^\top$ and by $(-\mu^{-1})$ for the components of $(\mathrm{d}x^{23},\mathrm{d}x^{31},\mathrm{d}x^{12})^\top$, considering the Euclidean metric in $\R^4.$

	Next, we derive the space-time vectorial wave equation from the 4D representation~\eqref{Einf:eqn:Maxwell4D} of Maxwell's equations.
	Since we assume that the domain is star-shaped with respect to a ball, the Poincar\'{e} lemma \cite[Theorem~4.1]{lang} is applicable. Hence, the closed form $F$ is exact and there is a potential $\mathcal A$, which is a $1$-form such that $\mathrm{d}\mathcal A = F.$
	From this, we also derive the following relations in the Euclidean metric 
	\begin{align*}
		\vec E &= -\partial_t \vec{A} +\nabla_x A_0,\\
		\vec B &= \cu_x\vec{A}.
	\end{align*}
	Here, the scalar function $A_0$ is the time component and the vector-valued function $\vec{A}:=(A_1,A_2,A_3)^\top$ is the spatial component of $\mathcal A$. If we insert this result into the second equation of \eqref{Einf:eqn:Maxwell4D}, combined with the third of \eqref{Einf:eqn:Maxwell4D}, we get the second-order differential equation
	\begin{equation}
		\label{Einf:eqn:PotentialEquation}
		\mathrm{d} \star^{\epsilon,-\mu^{-1}} \mathrm{d} \mathcal A =  \mathcal{J}.
	\end{equation}	
	
    We emphasize that there is no uniqueness of the potential $\mathcal A$. Indeed,  adding to $\mathcal A$ the exterior derivative of any $0$-form $\phi$ results again in a suitable potential $\tilde{\mathcal A}= \mathcal A+\mathrm{d}\phi$ which solves equation~\eqref{Einf:eqn:PotentialEquation}. This is equivalent to adding the space-time gradient  $\nabla_{(t,x)} \phi$ of a function $\phi \in H^1_0(Q)$ to $(A_0,\vec A)^\top$, with the usual Sobolev space $H^1_0(Q)$. In physics, this is called the \textit{gauge invariance}, which \textit{'is a manifestation of (the) nonobservability of $\mathcal{A}$'}, see \cite[p.~676]{jackson2001historical}. To make the potential $\mathcal A$ unique, we use gauging. A survey of different gauges can be found in \cite{jackson2001historical}. The choice of the gauge determines the resulting partial differential equation.
    In this paper, we deploy the Weyl gauge which is also named the temporal gauge, \cite{jackson2001historical}. The Weyl gauge is characterized by choosing $A_0=0$, i.e. the component of $\mathcal A$ in the time direction is zero. Applying it to \eqref{Einf:eqn:PotentialEquation} results in the vectorial wave equation.

	Before we state the vectorial wave equation in terms of the Euclidean metric, we introduce some notation that is used in this paper. First, for $d \in \{2,3\}$, the spatial bounded Lipschitz domain $\Omega \subset \R^d$ with boundary $\partial \Omega$ has the outward unit normal $\vec n_x \colon \, \partial \Omega \to \R^d$. Second, for a given terminal time $T>0$, we define the space-time cylinder $Q:=(0,T) \times \Omega \subset \R^{d+1}$ and its lateral boundary $\Sigma := [0,T]\times\partial\Omega \subset \R^{d+1}$. Next, we fix some notation, which differs for $d=2$ and $d=3.$ For this purpose, let $\vec f \colon \, Q \to \R^d$ be a sufficiently smooth, vector-valued function with components $f_\iota$, $\iota=1,\dots,d$, and $g \colon \, Q \to \R$ sufficiently smooth, scalar-valued function. We introduce the following notation:
	\begin{itemize}
	 \item \underline{$d=2$:} We define $\vec a \times \vec b = a_1 b_2 - a_2 b_1$ for vectors $\vec a = (a_1, a_2)^\top, \vec b=(b_1,b_2)^\top \in \R^2,$ and the spatial curl of $\vec f$ by the scalar function $\cu_x \vec f = \nabla_x \times \vec f = \partial_{x_1} f_2 - \partial_{x_2} f_1$ as well as the spatial curl of $g$ by the vector-valued function $\cu_x g=(\partial_{x_2} g, - \partial_{x_1} g)^\top.$
	 \item \underline{$d=3$:} In this case, the spatial curl of $\vec f$ is given by the vector-valued function
	\begin{equation*}
            \cu_x \, \vec f = \nabla_x \times \vec f = \left( \partial_{x_2} f_3 - \partial_{x_3} f_2, \partial_{x_3} f_1 - \partial_{x_1} f_3, \partial_{x_1} f_2 - \partial_{x_2} f_1 \right)^\top,
	\end{equation*}
	where $\times$ denotes the usual cross product of vectors in $\R^3.$
	\end{itemize}
    
	With this notation and the Weyl gauge, we rewrite equation~\eqref{Einf:eqn:PotentialEquation} into the vectorial wave equation for $d=2,3$ simply by inserting the definition of the exterior derivative in the Euclidean metric following \cite[Chapter~2]{ArnoldFalkWinther}. Hence, we want to find a function $\vec A \colon \overline{Q} \to \R^d$ such that
	\begin{equation}  \label{Einf:Maxwell}
		\left.
		\begin{array}{rcll}
			\partial_t \left( \epsilon \partial_t \vec A\right) + \cu_x \left(\mu^{-1} \cu_x \vec{A} \right) &=& \vec{j} &\text{ in } Q, \\
			\vec{A}(0,\cdot) &=& 0 &\text{ in } \Omega, \\
			\partial_t\vec{A}(0,\cdot) &=& 0 &\text{ in } \Omega, \\
			\trt\vec A&=& 0  &\text{ on } \Sigma,
		\end{array}
		\right \}
	\end{equation}
	where $\trt$ is the tangential trace operator, $\vec j \colon \, Q \to \R^d$ is a given current density, $\epsilon \colon \, \Omega \to \R^{d \times d}$ is a given permittivity, and $\mu \colon \, \Omega \to \R$ for $d=2$ and $\mu \colon \, \Omega \to \R^{3 \times 3}$ for $d=3$ is a given permeability.
	However, equation~\eqref{Einf:eqn:PotentialEquation} is a four-dimensional equation for a four-dimensional vector potential $(A_0,\underline{A})^\top$. Under the Weyl gauge, the fourth equation in \eqref{Einf:eqn:PotentialEquation} is $\di_x(\epsilon (\partial_t\vec{A}))=-\rho$  in $Q$. On the other hand, this equation holds
	as long as $\partial_t\underline{A}$ satisfies the initial condition $\di_x(\epsilon\partial_t\underline{A})(0,\cdot)=-\rho(0,\cdot)$ in $\Omega$, see \cite{HauserOhm2023} for more details. So, in case of $\partial_t\vec{A}(0,\cdot) = 0$ in $\Omega$ we are limited to examples where $\rho(0,\cdot) = 0$ in $\Omega$. However, using the same technique used for inhomogeneous Dirichlet boundary conditions in finite element methods for the Poisson's equation, see e.g.~\cite{ErnGuermond2020II}, we can extend the results of this paper to inhomogeneous initial data.	
	
	With this knowledge, we give an outline of this paper. In Section~\ref{Sec:FS}, we recall well-known function spaces, which are needed to state the trial and test spaces of the variational formulations. In Section~\ref{Sec:HT}, we introduce the modified Hilbert transformation which allows for a Galerkin--Bubnov formulation. Then, we state different variational formulations and their properties in Section \ref{Sec:VF}.
	Section~\ref{Sec:FES} is devoted to the finite element spaces. These approximation spaces are used for the space-time continuous Galerkin finite element discretizations in Section~\ref{Sec:FEM}, including investigations of a resulting CFL condition.
	Numerical examples for a two-dimensional spatial domain are presented in Section~\ref{Sec:Num}, which show the sharpness of the CFL condition and the different behavior of the numerical solutions when changing the projection of the right-hand side. Finally, we draw conclusions in Section~\ref{Sec:Zum}.

\section{Classical function spaces} \label{Sec:FS}
	
	For completeness, we recall classical function spaces, see \cite{Assous2018, ErnGuermond2020I, Monk2003, ZankDissBuch2020} for further details and references. For this purpose, let a bounded Lipschitz domain $\domain \subset \R^m$ for $m \in \N$ be fixed. We denote by $L^p(\domain)$, $1 \leq p \leq \infty$, the usual Lebesgue space with its norm $\norm{\cdot}_{L^p(\domain)}$, and by $H^k(\domain)$, $k \in \N$, the Sobolev space with Hilbertian norm $\norm{\cdot}_{H^k(\domain)}$. Further, the subspace $H^1_0(\domain) := \{ v \in H^1(\domain): v_{|\partial \domain}=0 \} \subset H^1(\domain)$ is endowed with the Hilbertian norm $\norm{\cdot}_{H^1_0(\domain)} := \abs{\cdot}_{H^1(\domain)} := \norm{\nabla_x (\cdot) }_{L^2(\domain)},$ which is actually a norm in $H^1_0(\domain)$ due to the Poincaré inequality.
	For the interval $\domain = (0,T)$, we write $L^2(0,T) := L^2((0,T))$, $H^k(0,T) := H^k((0,T))$, and the subspaces
	\begin{align*}
		H^1_{0,}(0,T) &:= \{ v \in H^1(0,T) : v(0) = 0 \} \subset H^1(0,T), \\
		H^1_{,0}(0,T) &:= \{ v \in H^1(0,T) : v(T) = 0 \} \subset H^1(0,T) 
	\end{align*}
	are equipped with the Hilbertian norm $\abs{\cdot}_{H^1(0,T)} := \norm{\partial_t (\cdot) }_{L^2(0,T)}$, which again, is actually a norm in $H^1_{0,}(0,T)$ and $H^1_{,0}(0,T)$ due to Poincaré inequalities. By interpolation, we introduce $H^s_{0,}(0,T) := [H^1_{0,}(0,T),L^2(0,T)]_s$ and $H^s_{,0}(0,T) := [H^1_{,0}(0,T),L^2(0,T)]_s$
	for $s \in [0,1]$.
	Further, the aforementioned spaces are generalized from real-valued functions $v \colon \, \domain \to \R$ to vector-valued functions $v \colon \, \domain \to X$ for a Hilbert space $X$ with inner product $(\cdot, \cdot)_X$. In particular, for $X = \R^n$ with $n \in \N$, the space $L^2(\domain;\R^n)$ is the usual Lebesgue space for vector-valued functions $\vec v \colon \, \domain \to \R^n$ endowed with the inner product $(\vec v,\vec w)_{L^2(\domain)} := (\vec v, \vec w)_{L^2(\domain;\R^n)}:= \int_\domain \vec v(x) \cdot \vec w(x) \mathrm dx$ for $\vec v, \vec w \in L^2(\domain;\R^n)$ and the induced norm $\norm{\cdot}_{L^2(\domain)} := \norm{\cdot}_{L^2(\domain; \R^n)} := \sqrt{(\cdot,\cdot)_{L^2(\domain)}}$. 
	
	With this notation, for a bounded Lipschitz domain $\Omega \subset \R^d$ with $d \in \{2,3\}$, we define the vector-valued Sobolev spaces
	\begin{equation*}
		H(\di;\Omega) := \left\{ \vec v \in L^2(\Omega;\R^d) : \di_x \vec v \in L^2(\Omega;\R) \right\}
	\end{equation*}
	endowed with the Hilbertian norm $\norm{\vec v}_{H(\di;\Omega)} := \left(  \norm{\vec v}_{L^2(\Omega)}^2 + \norm{\di_x \vec v}_{L^2(\Omega)}^2 \right)^{1/2},$ where $\di_x$ is the (distributional) divergence. Similarly, with the (distributional) curl, we set
	\begin{equation*}
		H(\cu;\Omega) := \begin{cases}
			\left\{ \vec v \in L^2(\Omega;\R^2) :  \cu_x \vec v \in L^2(\Omega; \R) \right\}, & d=2, \\
			\left\{ \vec v \in L^2(\Omega;\R^3) :  \cu_x \vec v \in L^2(\Omega; \R^3) \right\}, & d=3,
		\end{cases}
	\end{equation*}
	equipped with their natural Hilbertian norm $\norm{\vec v}_{H(\cu;\Omega)} := \left(  \norm{\vec v}_{L^2(\Omega)}^2 + \norm{\cu_x \vec v}_{L^2(\Omega)}^2 \right)^{1/2}.$

	Next, we introduce the tangential trace operator $\trt$ for $d \in \{2,3\}$, see \cite{Assous2018}, \cite[Section~4.3]{ErnGuermond2020I}, \cite[Theorem~2.11]{GiraultRaviart1986} and \cite[Subsection~3.5.3]{Monk2003} for details. For this purpose, we denote by $(\cdot)_{|\partial \Omega}$ the usual trace of a function as a linear, surjective mapping from $H^1(\Omega;\R^m)$ onto $H^{1/2}(\partial \Omega;\R^m)$ with $m \in \N$. Here, $H^{1/2}(\partial \Omega;\R^m)$, $m \in \N$, is the fractional-order Sobolev space equipped with the Sobolev–Slobodeckij norm, see \cite[Subsection~2.2.2]{ErnGuermond2020I}, and $H^{-1/2}(\partial \Omega;\R^m)$ is its dual space.
	To define the tangential trace operator $\trt$, we distinguish two cases:
	\begin{itemize}
            \item \underline{$d=2$:} For a function $\vec w \in H^1(\Omega;\R^2)$, we define the pointwise trace $\widetilde{\trt} \vec w = \vec w_{|\partial \Omega} \cdot \vec \tau_x$, where $\vec \tau_x $ is the unit tangent vector satisfying $\vec \tau_x \cdot \vec n_x = 0$. Then, the continuous mapping $\trt \colon \, H(\cu;\Omega) \to H^{-1/2}(\partial \Omega;\R)$ is the unique extension of $\widetilde{\trt}.$
            \item \underline{$d=3$:} Analogously, for a function $\vec w \in H^1(\Omega;\R^3)$, we set the pointwise trace as the function $\widetilde{\trt} \vec w = \vec w_{|\partial \Omega} \times \vec n_x$, and the continuous mapping $\trt \colon \, H(\cu;\Omega) \to H^{-1/2}(\partial \Omega;\R^3)$ as the unique extension of $\widetilde{\trt}$.
	\end{itemize}

    Having defined the tangential trace operator $\trt$ for $d \in \{2,3\}$, we introduce the closed subspace
	\begin{equation*}
		H_0(\cu;\Omega) := \left\{ \vec v \in H(\cu;\Omega): \trt \vec v = 0 \right\} \subset H(\cu;\Omega)
	\end{equation*}
	with the Hilbertian norm $\norm{\cdot}_{H_0(\cu;\Omega)} := \norm{\cdot}_{H(\cu;\Omega)}.$ Last, we recall that the set $C_0^\infty(\Omega;\R^d)$ of smooth functions with compact support is dense in $H_0(\cu;\Omega)$ with respect to $\norm{\cdot}_{H_0(\cu;\Omega)}$, see \cite[Theorem~2.12]{GiraultRaviart1986}.

\section{Modified Hilbert transformation}  \label{Sec:HT}
	
	In this section, we introduce the modified Hilbert transformation
	${\mathcal{H}}_T$ as developed in \cite{SteinbachZankETNA2020, ZankDissBuch2020}, and state its main properties,
	see also \cite{SteinbachMissoni2022, SteinbachZankJNUM2021, ZankCMAM2021, ZankInt2022}. The 
	modified Hilbert transformation acts in time only. Hence, in this section, we consider functions $v \colon \, (0,T) \to \R$, where a generalization to functions in $(t,x)$ is straightforward because of the tensor product structure of the domain $Q$.
	
	For $v \in L^2(0,T)$, we consider the Fourier series expansion
	\[
	v(t) = \sum\limits_{k=0}^\infty v_k \sin \left( \left(
	\frac{\pi}{2} + k \pi \right) \frac{t}{T} \right), \quad
	v_k := \frac{2}{T} \int_0^T v(t) \, \sin \left( \left(
	\frac{\pi}{2} + k \pi \right) \frac{t}{T} \right) \, \mathrm dt,
	\]
	and we define the modified Hilbert transformation ${\mathcal{H}}_T$ as
	\begin{equation}\label{HT:Def_H_T}
		({\mathcal{H}}_Tv)(t) = \sum\limits_{k=0}^\infty v_k \cos \left( \left(
		\frac{\pi}{2} + k \pi \right) \frac{t}{T} \right), \quad t \in (0,T).
	\end{equation}
	Note that the functions $t \mapsto \sin \left( \left(\frac{\pi}{2} + k \pi \right) \frac{t}{T} \right)$, $k \in \N_0$, form an orthogonal basis of $L^2(0,T)$ and $H^1_{0,}(0,T)$, whereas the functions $t \mapsto \cos \left( \left(\frac{\pi}{2} + k \pi \right) \frac{t}{T} \right)$, $k \in \N_0$, form an orthogonal basis of $L^2(0,T)$ and $H^1_{,0}(0,T)$. Hence, the mapping ${\mathcal{H}}_T \colon \, H^s_{0,}(0,T) \to H^s_{\,,0}(0,T)$ is an isomorphism for $s \in [0,1]$, where the inverse is the $L^2(0,T)$-adjoint,
	i.e.
	\begin{equation}  \label{HT:HTMinus1}
		( {\mathcal{H}_T}v , w )_{L^2(0,T)} =
		( v , {\mathcal{H}}_T^{-1} w )_{L^2(0,T)}
	\end{equation}
    for all $ v,w \in L^2(0,T)$. In addition, the relations
	\begin{align}
		( v , {\mathcal{H}}_T v )_{L^2(0,T)}
		&> 0 &\text{for }& 0\neq v \in H^s_{0,}(0,T), 0 < s \leq 1, \label{HT:positiv} \\
		\partial_t {\mathcal{H}}_T v
		&= - {\mathcal{H}}_T^{-1} \partial_t v \, \text{ in } L^2(0,T)
		&\text{for }& v \in H^1_{0,}(0,T)  \label{HT:vertauschen}
	\end{align}
	hold true. For the proofs of these aforementioned properties, we refer to
	\cite{SteinbachMissoni2022, SteinbachZankETNA2020, SteinbachZankJNUM2021, ZankDissBuch2020,   ZankCMAM2021, ZankInt2022}.
	Furthermore, the modified Hilbert transformation \eqref{HT:Def_H_T}
	allows a closed representation \cite[Lemma 2.8]{SteinbachZankETNA2020} as
	Cauchy principal value integral, i.e. for $v \in L^2(0,T)$,
	\[
	({\mathcal{H}}_T v)(t) =  {\mathrm  {v.p.}} \int_0^T \frac{1}{2T}
	\left(
	\frac{1}{\sin \frac{\pi(s+t)}{2T}} +
	\frac{1}{\sin \frac{\pi(s-t)}{2T}}
	\right) v(s) \, \mathrm ds, \quad t \in (0,T).
	\]
	Further integral representations of $\mathcal H_T$ are contained in \cite{SteinbachZankJNUM2021, ZankInt2022}.

\section{Space-time variational formulations}  \label{Sec:VF}
	
	In this section, we state space-time variational formulations of the vectorial wave equation~\eqref{Einf:Maxwell}.
	For this purpose, we recall a variational setting with different trial and test spaces as well as establish the tools to design a new space-time variational formulation of the vectorial wave equation~\eqref{Einf:Maxwell} with equal trial and test spaces.
	The latter space-time variational formulation is derived by using the modified Hilbert transformation $\mathcal H_T$ for the temporal part as introduced in Section~\ref{Sec:HT}.

	To deploy the existence and uniqueness result of \cite{HauserOhm2023} to the stated variational setting, we make the following assumptions, which are assumed for the rest of this work.
	\begin{assumption}[{\cite[Assumption~1]{HauserOhm2023}}] \label{VF:Voraussetzung}
		Let the spatial domain $\Omega \subset \R^d$, $d=2,3$, be given such that
		\begin{itemize}
			\item $\Omega$ is a bounded Lipschitz domain,
			\item and $Q := (0,T) \times \Omega$ is star-shaped with respect to a ball $B$, i.e. the convex hull of $B$ and each point in $Q$ is contained in $Q$.
		\end{itemize}
		Moreover, let $\vec j$, $\epsilon$ and $\mu$ be given functions, which satisfy:
		\begin{itemize}
			\item The function $\vec j \colon \, Q \to \R^d$ is in $L^2(Q;\R^d)$.
			\item The permittivity $\epsilon \in L^\infty(\Omega; \R^{d \times d})$ is symmetric, and uniformly positive definite, i.e. 
			\begin{equation*}
				\essinf_{x\in \Omega} \inf_{0 \ne \xi\in \R^d} \frac{\xi^\top \epsilon(x) \xi}{\xi^\top\xi} > 0.
			\end{equation*}
			\item For $d=2$, the permeability $\mu\in L^\infty(\Omega;\R)$ satisfies 
			\begin{equation*}
				 \essinf_{x\in \Omega} \mu(x)  > 0.
			\end{equation*}
			For $d=3$, the permeability $\mu \in L^\infty(\Omega; \R^{3 \times 3})$  is symmetric, and uniformly positive definite, i.e. 
			\begin{equation*}
				 \essinf_{x\in \Omega} \inf_{0 \ne \xi\in \R^3} \frac{\xi^\top \mu(x) \xi}{\xi^\top\xi} > 0.
			\end{equation*}
		\end{itemize}
	\end{assumption}
	Note that in Assumption~\ref{VF:Voraussetzung}, the given functions $\vec j$, $\epsilon$ and $\mu$ are real-valued and that the functions $\epsilon$, $\mu$ do not depend on the temporal variable $t.$ 
	Further, to handle the case that $Q$ is not star-shaped with respect to a ball, we introduce the vectorial wave equation~\eqref{Einf:Maxwell} in an extended space-time domain $\widetilde Q$, which is star-shaped with respect to a ball and fulfills $Q \subset \widetilde Q$. Here, the functions $\vec j$, $\epsilon$ and $\mu$ have to be extended appropriately by considering the material, e.g. assuming $\widetilde Q\backslash \overline{Q}$ is filled with air.

\subsection{Different trial and test spaces}

In this subsection, we derive a space-time variational formulation of the vectorial wave equation~\eqref{Einf:Maxwell}, where the trial and test spaces are different. With the aforementioned assumptions on the functions $\epsilon$ and $\mu$ the function spaces $L^2(\Omega;\R^d)$ and $H_0(\cu;\Omega)$ are endowed with the inner products
\begin{equation*}
	(\vec v, \vec w)_{L^2_\epsilon(\Omega)} := \int_\Omega  \big( \epsilon(x) \vec v(x) \big) \cdot \vec w(x) \mathrm dx, \quad \vec v, \, \vec w \in L^2(\Omega;\R^d),
\end{equation*}
and 
\begin{equation*}
	(\vec v, \vec w)_{H_{0,\epsilon,\mu}(\cu;\Omega)} := (\vec v, \vec w)_{L^2_\epsilon(\Omega)} +  (\vec v, \vec w)_{H_{0,\mu}(\cu;\Omega)}, \quad \vec v, \, \vec w \in H_0(\cu;\Omega),
\end{equation*}
respectively, where
\begin{equation*}
	(\vec v, \vec w)_{H_{0,\mu}(\cu;\Omega)} := \int_\Omega \big( \mu(x)^{-1} \cu_x \vec v(x) \big) \cdot \cu_x \vec w(x)  \mathrm dx, \quad \vec v, \, \vec w \in H_0(\cu;\Omega).
\end{equation*}
We denote the induced norms by $\norm{\cdot}_{L^2_\epsilon(\Omega)}$, $\norm{\cdot}_{H_{0,\epsilon,\mu}(\cu;\Omega)}$ and the seminorm by $\abs{\cdot}_{H_{0,\mu}(\cu;\Omega)} := \sqrt{(\cdot, \cdot)_{H_{0,\mu}(\cu;\Omega)}}$. Further, the norm $\norm{\cdot}_{L^2_\epsilon(\Omega)}$ is equivalent to $\norm{\cdot}_{L^2(\Omega)}$ as well as 
the norm $\norm{\cdot}_{H_{0,\epsilon,\mu}(\cu;\Omega)}$ is equivalent to $\norm{\cdot}_{H_0(\cu;\Omega)}$.

Next, we define the space-time Sobolev spaces, which are used for the variational formulation. We consider
\begin{align}
	H^{\cu;1}_{0;0,}(Q) :=& L^2(0,T;H_0(\cu;\Omega))\cap H^1_{0,}(0,T;L^2(\Omega;\R^d)), \label{VF:Ansatzraum} \\
	H^{\cu;1}_{0;,0}(Q) :=& L^2(0,T;H_0(\cu;\Omega))\cap H^1_{,0}(0,T;L^2(\Omega;\R^d)), \label{VF:Testraum}
\end{align}
which are endowed with the Hilbertian norm
\begin{align*}
	\abs{\vec v}_{H^{\cu;1}(Q)} := \norm{\vec v}_{H^{\cu;1}_{0;0,}(Q)} &:= \norm{\vec v}_{H^{\cu;1}_{0;,0}(Q)} \\
	&:= \left( \int_0^T \norm{ \partial_t \vec v(t,\cdot) }_{{L^2_\epsilon(\Omega)}}^2 \mathrm dt + \int_0^T \abs{\vec v(t,\cdot)}_{H_{0,\mu}(\cu;\Omega)}^2 \mathrm dt \right)^{1/2}.
\end{align*}
Note that $\abs{\cdot}_{H^{\cu;1}(Q)}$ is actually a norm in $H^{\cu;1}_{0;0,}(Q)$ and $H^{\cu;1}_{0;,0}(Q)$, since the Poincaré inequality
\begin{equation*}
	\norm{\vec v}_{L^2(Q)} \leq \frac{2 T}{\pi} \norm{\partial_t \vec v}_{L^2(Q)}
\end{equation*}
holds true for all $\vec v \in  H^1_{0,}(0,T;L^2(\Omega;\R^d))$ and for all $\vec v \in  H^1_{,0}(0,T;L^2(\Omega;\R^d))$. This Poincaré inequality follows from
\begin{equation*}
	\norm{ \vec v }_{L^2(Q)}^2 = \sum_{\iota=1}^d \int_\Omega \norm{v_\iota(\cdot,x)}_{L^2(0,T)}^2 \mathrm dx \leq \frac{4 T^2}{\pi^2} \sum_{\iota=1}^d \int_\Omega \norm{\partial_t v_\iota(\cdot,x)}_{L^2(0,T)}^2 \mathrm dx =  \frac{4 T^2}{\pi^2} \norm{ \partial_t \vec v }_{L^2(Q)}^2
\end{equation*}
for all $\vec v = (v_1, \dots, v_d)^\top$ with $\vec v \in  H^1_{0,}(0,T;L^2(\Omega;\R^d))$ or $\vec v \in  H^1_{,0}(0,T;L^2(\Omega;\R^d))$, where \cite[Lemma~3.4.5]{ZankDissBuch2020} is applied.

With these space-time Sobolev spaces, we motivate a space-time variational formulation of the vectorial wave equation~\eqref{Einf:Maxwell}. For this purpose, we multiply the vectorial wave equation~\eqref{Einf:Maxwell} by a test function $\vec v$, integrate over the space-time domain $Q$ and then use integration by parts with respect to space and time, resulting in the space-time variational formulation:\\
Find $\vec A \in H^{\cu;1}_{0;0,} (Q)$ such that
\begin{equation} \label{VF:WaveVec}
	-(\epsilon\partial_t \vec A,\partial_t \vec v)_{L^2(Q)}+ (\mu^{-1} \cu_x \vec A,\cu_x \vec v)_{L^2(Q)} = (\vec j, \vec v)_{L^2(Q)}
\end{equation}
for all $\vec v \in H^{\cu;1}_{0;,0} (Q)$.\\
Note that the first initial condition $\vec A(0,\cdot) = 0$ is fulfilled in the strong sense, whereas the second initial condition $\partial_t \vec A(0,\cdot) = 0$ is incorporated in a weak sense in the right side of \eqref{VF:WaveVec}. On the other hand, the boundary condition $\trt\vec A = 0$ is satisfied in the strong sense with $\vec A(t,\cdot) \in H_0(\cu;\Omega)$ for almost all $t \in (0,T)$.

The unique solvability of the variational formulation~\eqref{VF:WaveVec} is proven in \cite{ HauserKurzSteinbach2023}, \cite[Theorem~3.8]{HauserDiss2021}, and in \cite{HauserOhm2023} as a special case of Theorem 2, which is summarized in the next theorem.
\begin{theorem} \label{VF:Thm:Vek}
	Let Assumption~\ref{VF:Voraussetzung} be satisfied. Then, a unique solution $\vec A \in	H^{\cu;1}_{0;0,}(Q)$ of the variational formulation~\eqref{VF:WaveVec} exists and the stability estimate
	\begin{equation*}
		\abs{\vec A}_{H^{\cu;1}(Q)} \leq T \norm{\vec j}_{L^2(Q)} 
	\end{equation*}    
	holds true.
\end{theorem}
Note that the trial and test spaces of the variational formulation~\eqref{VF:WaveVec} are different. To get equal trial and test spaces, the modified Hilbert transformation $\mathcal H_T$ of Section~\ref{Sec:HT} is used, which is investigated in the next subsection.

\subsection{Equal trial and test spaces}

The goal of this subsection is to state a space-time variational formulation of the vectorial wave equation~\eqref{Einf:Maxwell} with equal trial and test spaces. For this purpose, we extend the definition of the modified Hilbert transformation $\mathcal H_T$ of Section~\ref{Sec:HT} to the vector-valued functions in $ H^{\cu,1}_{0;0,}(Q)$ and $ H^{\cu,1}_{0;,0}(Q)$, where the extension is denoted again by $\mathcal H_T$. More precisely, we define the mapping $\mathcal H_T \colon \, L^2(Q; \R^d) \to L^2(Q; \R^d)$ by
\begin{equation} \label{VF:HT}
	({\mathcal{H}}_T \vec v)(t,x) :=  \sum_{i=-\infty}^\infty \sum_{k=0}^\infty v_{i,k} \cos \left( \Big( \frac{\pi}{2} + k\pi \Big) \frac{t}{T} \right) \vec \varphi_i(x), \quad (t,x) \in Q,
\end{equation}
where the given function $\vec v \in L^2(Q;\R^d)$ is represented by
\begin{equation} \label{VF:vFourier}
	\vec v(t,x) = \sum_{i=-\infty}^\infty \sum_{k=0}^\infty v_{i,k} \sin \left( \Big( \frac{\pi}{2} + k\pi \Big) \frac{t}{T} \right) \vec \varphi_i(x), \quad (t,x) \in Q.
\end{equation}
Here, the set $\{ \vec \varphi_i \in H_0(\cu;\Omega) : \, i \in \Z \}$ is
\begin{itemize}
  \item an orthonormal basis of $L^2(\Omega;\R^d)$ with respect to $(\cdot, \cdot)_{L^2_\epsilon(\Omega)}$,
  \item an orthogonal basis of $H_0(\cu;\Omega)$ with respect to $(\cdot, \cdot)_{H_{0,\epsilon,\mu}(\cu;\Omega)}$, and
  \item orthogonal with respect to $(\cdot, \cdot)_{H_{0,\mu}(\cu;\Omega)}$,
\end{itemize}
where we refer to \cite[Subsection~1.1.1]{HauserOhm2023} and \cite[Subsection~8.3.1.2]{Assous2018} for the construction of such a basis. The inverse operator $\mathcal H_T^{-1} \colon \, L^2(Q; \R^d) \to L^2(Q; \R^d)$ is defined by
\begin{equation*}
	({\mathcal{H}}_T^{-1} \vec w)(t,x) =  \sum_{i=-\infty}^\infty \sum_{k=0}^\infty w_{i,k} \sin \left( \Big( \frac{\pi}{2} + k\pi \Big) \frac{t}{T} \right) \vec \varphi_i(x), \quad (t,x) \in Q,
\end{equation*}
where the given function $\vec w \in L^2(Q;\R^d)$ is represented by
\begin{equation} \label{VF:wFourier}
	\vec w(t,x) = \sum_{i=-\infty}^\infty \sum_{k=0}^\infty w_{i,k} \cos \left( \Big( \frac{\pi}{2} + k\pi \Big) \frac{t}{T} \right) \vec \varphi_i(x), \quad (t,x) \in Q.
\end{equation}

Next, we prove properties of this definition of the modified Hilbert transformation $\mathcal H_T$. These properties are then used for the derivation of a Galerkin--Bubnov method for the space-time variational formulation of the vectorial wave equation~\eqref{Einf:Maxwell}. Recall the definition of the vector-valued space-time Sobolev spaces $H^{\cu;1}_{0;0,}(Q)$, $H^{\cu;1}_{0;,0}(Q)$ defined in \eqref{VF:Ansatzraum}, \eqref{VF:Testraum}.

\begin{lemma} \label{VF:lem:HTIsometrie}
	The mapping ${\mathcal{H}}_T \colon \, H^{\cu;1}_{0;0,} (Q) \to H^{\cu;1}_{0;,0} (Q)$ is bijective and norm preserving, i.e. $\abs{\vec v}_{H^{\cu;1}(Q)} = \abs{\mathcal H_T \vec v}_{H^{\cu;1}(Q)}$ for all $ \vec v \in H^{\cu;1}_{0;0,} (Q).$
\end{lemma}
\begin{proof}
	We follow the arguments in \cite[Subsection~3.4.5]{ZankDissBuch2020}. Let $\vec v \in H^{\cu;1}_{0;0,} (Q)$ be fixed with the Fourier representations~\eqref{VF:vFourier}. The relation $\mathcal H_T \vec v \in H^{\cu;1}_{0;,0} (Q)$ and the bijectivity follow by definition. To prove the equality $\abs{\vec v}_{H^{\cu;1}(Q)} = \abs{\mathcal H_T \vec v}_{H^{\cu;1}(Q)}$, we use the representations \eqref{VF:HT}, \eqref{VF:vFourier} to calculate the functions $\partial_t \mathcal H_T \vec v$, $\cu_x \mathcal H_T \vec v$, $\partial_t \vec v$, $\cu_x \vec v$ as Fourier series. Plugging them into the norms $\abs{\vec v}_{H^{\cu;1}(Q)},$ $\abs{\mathcal H_T \vec v}_{H^{\cu;1}(Q)}$, and deploying the orthogonality relations of the basis $\{ \vec \varphi_i \in H_0(\cu;\Omega) : \, i \in \Z \}$ yield the assertion.
\end{proof}

As $\mathcal H_T$ acts only with respect to time, the properties~\eqref{HT:HTMinus1}, \eqref{HT:vertauschen} of Section~\ref{Sec:HT} remain valid. For completeness, we prove the following lemma.
\begin{lemma}
	The relations
	\begin{equation} \label{VF:HTMinus1}
		( \epsilon {\mathcal{H}_T} \vec v , \vec w )_{L^2(Q)} = (\epsilon \vec v , {\mathcal{H}}_T^{-1} \vec w )_{L^2(Q)}
	\end{equation}
	for all $\vec v \in L^2(Q;\R^d)$ and $ \vec w \in L^2(Q;\R^d)$, and 
	\begin{equation} \label{VF:HTvertauschen}
		\mathcal H_T^{-1}\partial_t \vec v = - \partial_t \mathcal H_T \vec v
	\end{equation}
	in $L^2(Q;\R^d)$  for all  $\vec v \in H^1_{0,}(0,T;L^2(\Omega;\R^d))$ are true.
\end{lemma}
\begin{proof}
	First, we prove property~\eqref{VF:HTMinus1}. Let $\vec v \in L^2(Q;\R^d)$ and $\vec w \in L^2(Q;\R^d)$ be fixed with the Fourier representations~\eqref{VF:vFourier} and \eqref{VF:wFourier}, respectively. Using the orthogonality properties of the involved basis functions and the symmetry of $\epsilon$ as stated in Assumption~\ref{VF:Voraussetzung} yield
	\begin{equation*}
		( \epsilon {\mathcal{H}_T} \vec v , \vec w )_{L^2(Q)} = \frac T 2 \sum_{i=-\infty}^\infty \sum_{k=0}^\infty v_{i,k} w_{i,k} = (\epsilon \vec v , {\mathcal{H}}_T^{-1} \vec w )_{L^2(Q)},
	\end{equation*}
	and thus, property~\eqref{VF:HTMinus1}.
	
	Second, to prove property~\eqref{VF:HTvertauschen}, fix again $\vec v \in H^1_{0,}(0,T;L^2(\Omega;\R^d))$ with the Fourier representations~\eqref{VF:vFourier}, which converges in $H^1_{0,}(0,T;L^2(\Omega;\R^d))$. Thus, the Fourier series
	\begin{equation*}
		\mathcal H_T^{-1}\partial_t \vec v(t,x) = \frac 1 T \sum_{i=-\infty}^\infty \sum_{k=0}^\infty v_{i,k} \Big( \frac{\pi}{2} + k\pi \Big) \sin \left( \Big( \frac{\pi}{2} + k\pi \Big) \frac{t}{T} \right) \vec \varphi_i(x), \quad (t,x) \in Q,
	\end{equation*}  
	is convergent with respect to $L^2(Q;\R^d)$. Analogously, as the Fourier series~\eqref{VF:HT} converges in $H^1_{,0}(0,T;L^2(\Omega;\R^d))$, the series
	\begin{equation*}
		\partial_t \mathcal H_T \vec v(t,x) = -\frac 1 T \sum_{i=-\infty}^\infty \sum_{k=0}^\infty v_{i,k} \Big( \frac{\pi}{2} + k\pi \Big) \sin \left( \Big( \frac{\pi}{2} + k\pi \Big) \frac{t}{T} \right) \vec \varphi_i(x), \quad (t,x) \in Q,
	\end{equation*}
	is also convergent in $L^2(Q;\R^d)$. Comparing these Fourier representations gives the equality $\mathcal H_T^{-1}\partial_t \vec v = - \partial_t \mathcal H_T \vec v$ in $L^2(Q;\R^d)$.
\end{proof}

Last, we need the following result.
\begin{lemma} \label{VF:lem:Vertauschen}
	For $\vec v \in L^2(Q;\R^d)$, the equality
	\begin{equation*}
		\mathcal H_T (\epsilon \vec v) = \epsilon \mathcal H_T \vec v \quad \text{ in }L^2(Q;\R^d) 
	\end{equation*}
	holds true.
\end{lemma}
\begin{proof}
	Let $\vec v \in L^2(Q;\R^d)$ be a fixed function. With the normalized eigenfunction $\psi_k(t) = \sqrt{\frac{2}{T}} \sin \left( \Big( \frac{\pi}{2} + k\pi \Big) \frac{t}{T} \right)$, we have the representations
	\begin{align*}
		(\epsilon\vec v)(t,x) &= \sum_{i=-\infty}^\infty \sum_{k=0}^\infty (\epsilon \epsilon \vec v, \psi_k   \vec \varphi_i)_{L^2(Q)} \psi_k(t) \vec \varphi_i(x), \quad (t,x) \in Q, \\
		\epsilon(x) \vec v(t,x) &= \sum_{i=-\infty}^\infty \sum_{k=0}^\infty (\epsilon \vec v, \psi_k   \vec \varphi_i)_{L^2(Q)} \psi_k(t) \epsilon(x) \vec \varphi_i(x), \quad (t,x) \in Q,
	\end{align*}
	which lead to the equality
	\begin{equation*}
		\forall k \in \N_0: \quad \sum_{i=-\infty}^\infty  (\epsilon \epsilon \vec v, \psi_k   \vec \varphi_i)_{L^2(Q)} \vec \varphi_i(x) = \sum_{i=-\infty}^\infty (\epsilon \vec v, \psi_k   \vec \varphi_i)_{L^2(Q)} \epsilon(x) \vec \varphi_i(x)
	\end{equation*}
	for almost all $x \in \Omega$. Then we derive
	\begin{align*}
		\epsilon(x) ({\mathcal{H}}_T \vec v)(t,x) &= \sqrt{\frac{T}{2}}  \sum_{k=0}^\infty \cos \left( \Big( \frac{\pi}{2} + k\pi \Big) \frac{t}{T} \right) \sum_{i=-\infty}^\infty(\epsilon \vec v, \psi_k   \vec \varphi_i)_{L^2(Q)}  \epsilon(x) \vec \varphi_i(x) \\
		&= \sqrt{\frac{T}{2}}  \sum_{k=0}^\infty \cos \left( \Big( \frac{\pi}{2} + k\pi \Big) \frac{t}{T} \right) \sum_{i=-\infty}^\infty  (\epsilon \epsilon \vec v, \psi_k   \vec \varphi_i)_{L^2(Q)} \vec \varphi_i(x)  = \mathcal H_T ( \epsilon \vec v) (t,x)
	\end{align*}
	for almost all $(t,x) \in Q$, and hence, the assertion.
\end{proof}

Next, we derive the Galerkin--Bubnov formulation. With the mapping ${\mathcal{H}}_T \colon \, H^{\cu;1}_{0;0,} (Q) \to H^{\cu;1}_{0;,0} (Q)$, the variational formulation~\eqref{VF:WaveVec} is equivalent to:\\
Find $\vec A \in H^{\cu;1}_{0;0,} (Q)$ such that
\begin{equation}\label{VF:WaveVecHTv}
	-(\epsilon\partial_t \vec A,\partial_t \mathcal H_T \vec v)_{L^2(Q)}+ (\mu^{-1} \cu_x \vec A,\cu_x \mathcal H_T \vec v)_{L^2(Q)} = (\vec j, \mathcal H_T \vec v)_{L^2(Q)}
\end{equation}
for all $\vec v \in H^{\cu;1}_{0;0,} (Q).$\\
We rewrite the variational formulation~\eqref{VF:WaveVecHTv} using the properties~\eqref{VF:HTMinus1}, \eqref{VF:HTvertauschen} to get:\\
Find $\vec A \in H^{\cu;1}_{0;0,} (Q)$ such that
\begin{equation}
	\label{VF:WaveVecHT}
	(\epsilon \mathcal H_T \partial_t \vec A,\partial_t \vec v)_{L^2(Q)}+ (\mu^{-1} \cu_x \vec A,\cu_x \mathcal H_T \vec v)_{L^2(Q)} = (\vec j, \mathcal H_T \vec v)_{L^2(Q)}
\end{equation}
for all $ \vec v \in H^{\cu;1}_{0;0,} (Q),$ where also Lemma~\ref{VF:lem:Vertauschen} is applied. The variational formulation~\eqref{VF:WaveVecHT} has a unique solution, as the equivalent variational formulation~\eqref{VF:WaveVec} is uniquely solvable, see Theorem~\ref{VF:Thm:Vek}, and due to the fact that ${\mathcal{H}}_T \colon \, H^{\cu;1}_{0;0,} (Q) \to H^{\cu;1}_{0;,0} (Q)$ is an isometry, see Lemma~\ref{VF:lem:HTIsometrie}.

\section{Finite element spaces} \label{Sec:FES}

As both variational formulations, the Galerkin--Petrov formulation \eqref{VF:WaveVec} and the Galerkin--Bubnov formulation \eqref{VF:WaveVecHT}, are uniquely solvable, we investigate the discrete counterparts of these formulations. For that purpose, we first introduce the finite element spaces, which are used for discretizations of the variational formulations \eqref{VF:WaveVec} and \eqref{VF:WaveVecHT}. For comparison reasons, we aim at a similar notation as in \cite[Subsection~3.1]{HauserOhm2023}, where, however, piecewise quadratic instead of piecewise linear functions are used for the temporal discretizations.

Let us start with the spatial discretization. In this section, we assume that the bounded Lipschitz domain $\Omega \subset \R^d$ is polygonal for $d=2$, or polyhedral for $d=3$. Let $\nu \in \N_0$ be the refinement level and $( \mathcal T^x_\nu )_{\nu \in \N_0}$ the corresponding mesh sequence of admissible decompositions $\mathcal T^x_\nu$ of the spatial domain $\Omega$. On each level $\nu$ we decompose the spatial domain $\Omega$ into $N^x:=N^x_\nu$ triangles ($d=2$) or tetrahedra ($d=3$) denoted by $\omega_\iota \subset \R^d$ for $\iota=1,\dots, N^x$, satisfying $\overline{\Omega} = \bigcup_{\iota=1}^{N^x} \overline{\omega_{\iota}}.$ Note that also other types of elements $\omega_\iota$, e.g. rectangles ($d=2$), are possible. We emphasize that we define the local mesh sizes by $h_{x,\iota} = \left( \int_{\omega_\iota} \mathrm dx \right)^{1/d}$, $\iota=1,\dots, N^x$, since this particular choice occurs in the context of the CFL condition of the proposed Galerkin--Petrov method. Further, we define $h_x := h_{x,\max}(\mathcal T^x_\nu) := \max_{\iota=1,\dots, N^x} h_{x,\iota}$, and $h_{x,\min}(\mathcal T^x_\nu) := \min_{\iota=1,\dots, N^x} h_{x,\iota}$. 

For the investigations of CFL conditions, we consider for the rest of the paper a sequence $( \mathcal T^x_\nu )_{\nu \in \N_0}$, satisfying a shape-regularity and a global quasi-uniformity. As these terms are not used in a uniform manner, we recall their definition. The mesh sequence $( \mathcal T^x_\nu )_{\nu \in \N_0}$ is
\begin{itemize}
  \item shape-regular if $\exists c_\mathrm F > 0: \forall \nu \in \N_0 : \forall \omega \in \mathcal T^x_\nu: \quad \sup_{x,y \in \overline{\omega} } \norm{x- y}_2 \leq c_\mathrm F\, r_\omega,$
  \item globally quasi-uniform if $\exists c_\mathrm G \geq 1: \forall \nu \in \N_0 :  \frac{h_{x,\max}(\mathcal T^x_\nu)}{h_{x,\min}(\mathcal T^x_\nu)} \leq c_\mathrm G.$
\end{itemize}
Here, $r_\omega >0$ is the inradius of the element $\omega$, and $\norm{x-y}_2$ is the Euclidean distance of the points $x,y$ in $\R^d.$  The constant $c_{\mathrm F}$ influences the conditional stability, namely a CFL condition. To derive this CFL condition, we also need an inverse inequality for the spatial curl operator. The inverse inequality on the other hand is affected by the global quasi-uniformity for finite elements in $H(\cu;\Omega)$.

With these conditions on the spatial mesh in mind, we introduce three spatial approximation spaces of vector-valued type, see \cite[Section~19.2]{ErnGuermond2020I}, \cite[Section~3.5.1]{logg2012automated} or \cite[Section~5.5]{Monk2003} for more details. For this purpose, for $p \in \N_0$, $n \in \{1,2,3\}$ and $K \subset \R^n$, the polynomials of total degree $p$ are denoted by $\mathbb P^p_n(K)$. First, we define the space of vector-valued piecewise constant functions
\begin{equation*}
	S_d^0(\mathcal T^x_\nu) := \Big\{ \vec v_{h_x} \in L^2(\Omega;\R^d) : \forall \omega \in \mathcal T^x_\nu: \vec v_{h_x|\omega} \in \mathbb P^0_d(\omega)^d \Big\} = \mathrm{span} \{\vec \psi_\ell^0 \}_{\ell=1}^{N_x^0}, \\
\end{equation*}
where the $N_x^0$ basis functions $\vec \psi_\ell^0$ are the componentwise characteristic functions with respect to the spatial elements. Second, we set the lowest-order Raviart--Thomas finite element space by
\begin{equation*}
	\mathcal{RT}^0(\mathcal T^x_\nu) := \left\{ \vec v_{h_x} \in H(\di;\Omega) : \, \forall \omega \in \mathcal T^x_\nu : \vec v_{h_x| \omega} \in \mathcal{RT}^0(\omega) \right\} = \mathrm{span} \{\vec \psi_\ell^{\mathcal{RT}} \}_{\ell=1}^{N_x^{\mathcal{RT}}},
\end{equation*}
where the $N_x^{\mathcal{RT}}$ basis functions $\vec \psi_\ell^{\mathcal{RT}}$ are attached to the edges for $d=2$ and the faces for $d=3$ of the spatial mesh $\mathcal T^x_\nu$ and
\begin{equation*}
	\mathcal{RT}^0(\omega) := \Big\{ \vec P \in \mathbb P^1_d(\omega)^d: \forall x \in \omega: \vec P(x) = \vec a + b x \text{ with } \vec a \in \R^d, b \in \R \Big\}
\end{equation*}
is the local polynomial space for $\omega \in \mathcal T^x_\nu$, see \cite[Section~14.1]{ErnGuermond2020I}. Last, we introduce the lowest-order Nédélec finite element space of the first kind by
\begin{equation*}
	\mathcal N_\mathrm{I}^0(\mathcal T^x_\nu) := \left\{ \vec v_{h_x} \in H(\cu;\Omega) : \, \forall \omega \in \mathcal T^x_\nu : \vec v_{h_x| \omega} \in \mathcal N_\mathrm{I}^0(\omega) \right\},
\end{equation*}
and its subspace having zero tangential trace by
\begin{equation*}
	\mathcal N_\mathrm{I,0}^0(\mathcal T^x_\nu) := \mathcal N_\mathrm{I}^0(\mathcal T^x_\nu) \cap H_0(\cu;\Omega) = \mathrm{span} \{\vec \psi_\ell^{\mathcal{N}} \}_{\ell=1}^{N_x^{\mathcal{N}}},
\end{equation*}
where the $N_x^{\mathcal{N}}$ basis functions $\vec \psi_\ell^{\mathcal{N}}$ are attached to the edges of the spatial mesh $\mathcal T^x_\nu$, and the local polynomial spaces for $\omega \in \mathcal T^x_\nu$ are given by \begin{equation*}
	\mathcal N_\mathrm{I}^0(\omega) := \Big\{ \vec P \in \mathbb P^1_2(\omega)^2: \forall (x_1,x_2) \in \omega: \vec P(x_1,x_2) = \vec a + b \cdot (-x_2,x_1)^\top \text{ with } \vec a \in \R^2, b \in \R \Big\}
\end{equation*}
for $d=2$, and
\begin{equation*}
	\mathcal N_\mathrm{I}^0(\omega) := \Big\{ \vec P \in \mathbb P^1_3(\omega)^3: \forall x \in \omega: \vec P(x) = \vec a + \vec b \times x
	\text{ with } \vec a \in \R^3, \vec b \in \R^3 \Big\}
\end{equation*}
for $d=3$, see \cite[Section~15.1]{ErnGuermond2020I}.

Next, we investigate temporal finite element spaces. For this purpose, we consider a sequence of meshes $\{\mathcal T^t_{\alpha}\}_{\alpha \in \N_0}$ in time. Here, $\alpha \in \N_0$ is the level of refinement and $\mathcal T^t_\alpha =  \{ \tau_l \}_{l=1}^{N^t_\alpha}$ is the mesh defined by the decomposition
\begin{equation*}
	0 = t_0 < t_1 < \dots < t_{N^t_\alpha-1} < t_{N^t_\alpha} = T
\end{equation*}
of the time interval $(0,T)$, where $N^t_\alpha=:N^t$ is the number of temporal elements $\tau_l = (t_{l-1}, t_l) \subset \R$, $l=1,\dots, N^t,$ with mesh sizes $h_{t,l} = t_l - t_{l-1}$, $l=1,\dots, N^t,$ and $h_t := \max_{l=1,\dots, N^t} h_{t,l}.$ First, the space of piecewise constant functions in time is given by
\begin{equation*}
	S^0(\mathcal T^t_\alpha) := \Big\{ v_{h_t} \in L^2(0,T): \forall l \in \{ 1, \dots, N^t \}: v_{h_t|\tau_l} \in \mathbb P^0_1(\tau_l)  \Big\} =  \mathrm{span} \{\varphi^0_l\}_{l=1}^{N^t}
\end{equation*}
with the elementwise characteristic functions $\varphi^0_l$ as basis functions. Second, we write the space of piecewise linear, globally continuous functions as
\begin{equation*}
	S^1(\mathcal T^t_\alpha) := \Big\{ v_{h_t} \in C[0,T]: \forall l \in \{ 1, \dots, N^t \}: v_{h_t|\overline{\tau_l}} \in \mathbb P^1_1(\overline{\tau_l})  \Big\}  = \mathrm{span} \{\varphi^1_l\}_{l=0}^{N^t},
\end{equation*}
where the hat function $\varphi^1_l$ is related to $t_l$, i.e. $\varphi^1_l(t_k)=0$ for $k \neq l$ and $\varphi^1_l(t_l)=1$. As we need also approximation spaces that incorporate initial or terminal conditions, we define the subspace
\begin{equation*}
	S_{0,}^1(\mathcal T^t_\alpha) := S^1(\mathcal T^t_\alpha) \cap H^1_{0,}(0,T) = \mathrm{span} \{\varphi^1_l\}_{l=1}^{N^t},
\end{equation*}
fulfilling the homogeneous initial condition, and the subspace
\begin{equation*}
	S_{,0}^1(\mathcal T^t_\alpha) := S^1(\mathcal T^t_\alpha) \cap H^1_{,0}(0,T) = \mathrm{span} \{\varphi^1_l\}_{l=0}^{N^t-1},
\end{equation*}
satisfying the homogeneous terminal condition.

With the spatial and temporal approximation spaces, we introduce the conforming subspaces
\begin{align}
	S_{0,}^1(\mathcal T^t_\alpha) \otimes\mathcal N_\mathrm{I,0}^0(\mathcal T^x_\nu) &\subset H^{\cu;1}_{0;0,}(Q), \label{FES:TPSpaceA0}  \\
    S_{,0}^1(\mathcal T^t_\alpha) \otimes\mathcal N_\mathrm{I,0}^0(\mathcal T^x_\nu) &\subset H^{\cu;1}_{0;,0}(Q),  \label{FES:TPSpaceE0}
\end{align}
see \eqref{VF:Ansatzraum}, \eqref{VF:Testraum}, where $\otimes$ denotes the Hilbert tensor-product. Thus, these space-time approximation spaces are related to the space-time mesh $\{ \tau_l \times \omega_\iota : \, l=1,\dots,N^t, \, \iota = 1,\dots,N^x \}.$ In other words, the space-time cylinder $Q \subset \R^{d+1}$ is decomposed into $N^t \cdot N^x$ space-time elements $\tau_l \times \omega_\iota \subset \R^{d+1}$ for $l=1,\dots,N^t$, $\iota=1,\dots,N^x$.

Last, to approximate the right-hand side of the vectorial wave equation~\eqref{Einf:Maxwell}, we consider two different $L^2(Q)$ projections of $\vec j\in L^2(Q; \R^d)$. In Section~\ref{Sec:Num} on the numerical results we investigate the influence of the two types of projection on the convergence rates. Let us take a look at the definition of the projections for a given $\vec g \in L^2(Q; \R^d)$. First, the projection onto the space of piecewise constant functions $\Pi_h^0 \colon \, L^2(Q; \R^d) \to S^0(\mathcal T^t_\alpha) \otimes S_d^0(\mathcal T^x_\nu)$ is defined as solution $\Pi_h^0 \vec g \in S^0(\mathcal T^t_\alpha) \otimes S_d^0(\mathcal T^x_\nu)$ such that
\begin{equation} \label{FES:L2Konst}
 \quad (\Pi_h^0 \vec g, \vec w_h)_{L^2(Q)} = (\vec g, \vec w_h)_{L^2(Q)}
\end{equation}
for all $\vec w_h \in S^0(\mathcal T^t_\alpha) \otimes S_d^0(\mathcal T^x_\nu).$ Second, the projection onto the space of piecewise linear functions in time and Raviart--Thomas functions in space $\Pi_h^{\mathcal{RT},1} \colon \, L^2(Q; \R^d) \to S^1(\mathcal T^t_\alpha) \otimes \mathcal{RT}^0(\mathcal T^x_\nu)$ is given as solution $\Pi_h^{\mathcal{RT},1} \vec g \in S^1(\mathcal T^t_\alpha) \otimes \mathcal{RT}^0(\mathcal T^x_\nu)$ such that
\begin{equation} \label{FES:L2RT}
\quad (\Pi_h^{\mathcal{RT},1} \vec g, \vec w_h)_{L^2(Q)} = (\vec g, \vec w_h)_{L^2(Q)}
\end{equation}
for all $\vec w_h \in S^1(\mathcal T^t_\alpha) \otimes \mathcal{RT}^0(\mathcal T^x_\nu).$ Note that both $L^2(Q)$ projections possibly have inhomogeneous Dirichlet, initial or terminal conditions.

\section{FEM for the vectorial wave equation} \label{Sec:FEM}

Following the introduction to the finite element spaces in the last section, we derive conforming space-time discretizations for the variational formulations~\eqref{VF:WaveVec}, \eqref{VF:WaveVecHT}, using the notation of Section~\ref{Sec:FES}. For this purpose, let the bounded Lipschitz domain $\Omega \subset \R^d$ be polygonal for $d=2$, or polyhedral for $d=3$. Additionally, we assume Assumption~\ref{VF:Voraussetzung} and the assumptions on the mesh stated in Section~\ref{Sec:FES}.

\subsection{Galerkin--Petrov FEM} \label{Sec:FEM:WithoutHT}

In this subsection, we state the conforming discretization of the variational formulation~\eqref{VF:WaveVec}, using the tensor-product spaces \eqref{FES:TPSpaceA0}, \eqref{FES:TPSpaceE0}, 
by a Galerkin--Petrov finite element method:\\
Find $\vec A_h \in S_{0,}^1(\mathcal T^t_\alpha) \otimes \mathcal N_\mathrm{I,0}^0(\mathcal T^x_\nu)$ such that
\begin{equation}  \label{FEM:DiskVFOhneHT}
	-(\epsilon \partial_t \vec A_h, \partial_t \vec w_h )_{L^2(Q)} + (\mu^{-1} \cu_x \vec A_h, \cu_x \vec w_h )_{L^2(Q)}  = (\Pi_h \vec j, \vec w_h)_{L^2(Q)}
\end{equation}
for all $\vec w_h \in S_{,0}^1(\mathcal T^t_\alpha) \otimes \mathcal N_\mathrm{I,0}^0(\mathcal T^x_\nu)$. Hence, we approximate the solution $\vec A \in H^{\cu;1}_{0;0,} (Q)$ of the variational formulation~\eqref{VF:WaveVec} by
\begin{equation} \label{FEM:DarstellungAh}
	\vec A(t,x) \approx \vec A_h(t,x) = \sum_{k=1}^{N^t} \sum_{\kappa=1}^{N_x^{\mathcal{N}}} \mathcal A_\kappa^k \varphi^1_k(t) \vec \psi^\mathcal{N}_\kappa(x), \quad (t,x) \in Q,
\end{equation}
with coefficients $\mathcal A_\kappa^k \in \R$. The operator $\Pi_h$ in \eqref{FEM:DiskVFOhneHT} is either the $L^2(Q)$ projection $\Pi_h^0$ onto the space of piecewise constant functions defined in \eqref{FES:L2Konst} or the $L^2(Q)$ projection $\Pi_h^{\mathcal{RT},1}$ onto piecewise linear functions in time and Raviart--Thomas functions in space given in \eqref{FES:L2RT}. The reason to replace the right-hand side $\vec j$ with $\Pi_h \vec j$ in \eqref{FEM:DiskVFOhneHT} is the comparison with the Galerkin--Bubnov finite element method~\eqref{FEM:WithHT:DiskVFMitHT}, using the modified Hilbert transformation $\mathcal H_T$. An alternative is the use of formulas for numerical integration applied to $(\vec j, \vec w_h)_{L^2(Q)}$. Note that such formulas for numerical integration have to be sufficiently accurate to preserve the convergence rates of the finite element method, see the numerical examples in Section~\ref{Sec:Num}.

Next, we use the representation \eqref{FEM:DarstellungAh} to rewrite the discrete variational formulation~\eqref{FEM:DiskVFOhneHT} into the equivalent linear system 
\begin{equation} \label{FEM:WithoutHT:LGS}
	(-A_{h_t}^1 \otimes M_{h_x}^{\mathcal{N}} +  M_{h_t}^1 \otimes A_{h_x}^{\mathcal{N}}) \vec{ \mathcal A }= \vec{\mathcal J}.
\end{equation}
Here, we order the degrees of freedom first in time then in space so that the coefficient vector in  \eqref{FEM:DarstellungAh} can be written as
\begin{equation} \label{FEM:WithoutHT:VectorAh}
	\vecA = ( \vecA^1, \vecA^2, \dots, \vecA^{N^t} )^\top \in \R^{N^t N_x^{\mathcal{N}}},
\end{equation}
where
\begin{equation*}
	\vecA^k = (\mathcal A_1^k, \mathcal A_2^k, \dots, \mathcal A_{N_x^{\mathcal{N}}}^k)^\top \in \R^{N_x^{\mathcal{N}}} \quad \text{ for } k =1,\dots, N^t.
\end{equation*}
Further, for the linear system~\eqref{FEM:WithoutHT:LGS}, we define the spatial matrices by 
\begin{equation} \label{FEM:WithoutHT:Ortsmatrizen}
	A_{h_x}^{\mathcal{N}}[\ell,\kappa] = (\mu^{-1}\cu_x \vec \psi^\mathcal{N}_\kappa, \cu_x \vec \psi^\mathcal{N}_{\ell})_{L^2(\Omega)}, \quad M_{h_x}^{\mathcal{N}}[\ell,\kappa] = (\epsilon \vec \psi^\mathcal{N}_\kappa, \vec \psi^\mathcal{N}_{\ell})_{L^2(\Omega)}
\end{equation}
for $\ell,\kappa=1,\dots,N_x^{\mathcal{N}}$ and the temporal matrices by
\begin{equation}  \label{FEM:WithoutHT:TimeMatrices}
	A_{h_t}^1[l,k] = (\partial_t \varphi^1_k, \partial_t \varphi^1_l)_{L^2(0,T)}, \quad M_{h_t}^1[l,k] = (\varphi^1_k, \varphi^1_l)_{L^2(0,T)}
\end{equation}
for  $l=0,\dots,N^t-1$, $k=1,\dots,N^t$ as well as the right-hand side by
\begin{equation*}
	\vec{\mathcal J} = ( \vec f^0, \vec f^1, \dots, \vec f^{N^t-1} )^\top \in \R^{N^t N_x^{\mathcal{N}}}
\end{equation*}
with $\vec f^l = (f_1^l, f_2^l, \dots, f_{N_x^{\mathcal{N}}}^l)^\top \in \R^{N_x^{\mathcal{N}}}$ for $l = 0,\dots, N^t-1$ and
\begin{equation} \label{FEM:WithoutHT:RSf}
	f_\ell^l = (\Pi_h \vec j,\varphi^1_l \vec \psi^\mathcal{N}_{\ell} )_{L^2(Q)} \quad \text{ for } l=0,\dots, N^t-1, \, \ell=1,\dots,N_x^{\mathcal{N}}.
\end{equation}

For the stability analysis of the linear system \eqref{FEM:WithoutHT:LGS}, we take a closer look at the system matrix. The system matrix of the linear system~\eqref{FEM:WithoutHT:LGS} is sparse and allows for a realization as a two-step method, due to the sparsity pattern of the temporal matrices $A_{h_t}^1$, $M_{h_t}^1$ in \eqref{FEM:WithoutHT:TimeMatrices}. Further, this sparsity pattern forms the basis of classical stability analysis of the Galerkin--Petrov finite element method~\eqref{FEM:DiskVFOhneHT} in the framework of time-stepping or finite difference methods. See \cite[Section~5]{SteinbachZankETNA2020} for the scalar wave equation, where the ideas can be extended to the vectorial wave equation, see the discussion in \cite[Section~3.5.2]{HauserDiss2021} or \cite{HauserKurzSteinbach2023}. Here, we skip details and only state that the Galerkin--Petrov finite element method ~\eqref{FEM:DiskVFOhneHT} is stable if the temporal mesh is uniform with mesh size $h_t$ and the CFL condition
\begin{equation} \label{FEM:WithoutHT:CFL}
	h_t < \sqrt{\frac{12}{c_\mathrm{I}}}h_x
\end{equation}
is satisfied with the constant $c_\mathrm{I} > 0$ of the spatial inverse inequality
\begin{equation} \label{FEM:WithoutHT:InverseUngleichung}
	\forall \vec v_{h_x} \in \mathcal N_\mathrm{I}^0(\mathcal T^x_\nu) : \quad \norm{\cu_x  \vec v_{h_x}}_{L^2(\Omega)}^2 \leq c_\mathrm{I} h_x^{-2} \norm{ \vec v_{h_x}}_{L^2(\Omega)}^2.
\end{equation}
Bounds for the constant $c_\mathrm{I} > 0$ are given in \cite[Lemma~A.2]{HauserDiss2021} or \cite{HauserOhm2023, HauserKurzSteinbach2023}.

As an example domain, which is also used in Section~\ref{Sec:Num}, we consider the unit square $\Omega = (0,1) \times (0,1) \subset \R^2$. Using uniform triangulations with isosceles right triangles as in Figure~\ref{FEM:WithoutHT:Fig:Quadrat} yields the constant $c_\mathrm I = 18$. This result is proven by adapting the arguments of the proof of \eqref{FEM:WithoutHT:InverseUngleichung}, given in \cite[Lemma~A.2]{HauserDiss2021} or \cite{HauserKurzSteinbach2023}, of the general situation to the case of isosceles right triangles. More precisely, in the proof of \cite[Lemma~A.2]{HauserDiss2021}, the matrix $J_l J_l^\top$ is diagonal for isosceles right triangles, where $J_l$ denotes the Jacobian matrix of the geometric mapping from the reference element to the physical one. Thus, the eigenvalues of $J_l J_l^\top$ are known explicitly. So, in this situation, the CFL condition~\eqref{FEM:WithoutHT:CFL} reads as
\begin{equation} \label{FEM:WithoutHT:CFLQuadrat}
	h_t < \sqrt{\frac{12 }{18}} h_x\approx 0.81649658\ h_x.
\end{equation}
Numerical examples indicate the sharpness of the CFL condition~\eqref{FEM:WithoutHT:CFLQuadrat}, see Subsection~\ref{Sec:Num:Ohne}.
\begin{figure}[ht]
	\begin{center}
		\begin{tikzpicture}
			\draw (0,0) rectangle(3,3);
			\draw (0,1.5) -- (1.5,3);
			\draw (1.5,0) -- (3,1.5);
			\draw (1.5,0) -- (1.5,1.5);
			\draw (0,0) -- (3,3);
			\draw (0,1.5) -- (1.5,1.5);
			\draw (1.5,3) -- (1.5,1.5);
			\draw (1.5,1.5) -- (3,1.5);
			\draw[->] (0,0) -- (3.4,0);
			\draw[->] (0,0) -- (0,3.4);
			\fill[black]  (3.5,0) circle [radius=0pt] node[below] {\footnotesize$x_{1}$}; 
			\fill[black]  (0,3.4) circle [radius=0pt] node[left] {\footnotesize$x_{2}$}; 
			\fill[black]  (-0.2,0) circle [radius=0pt] node[below] {\footnotesize$0$}; 
			\fill[black]  (3,0) circle [radius=0pt] node[below] {\footnotesize$1$}; 
			\fill[black]  (0,3) circle [radius=0pt] node[left] {\footnotesize$1$}; 
		\end{tikzpicture}
		\caption{Uniform triangulations of the unit square with isosceles right triangles.}\label{FEM:WithoutHT:Fig:Quadrat}
	\end{center}
\end{figure}
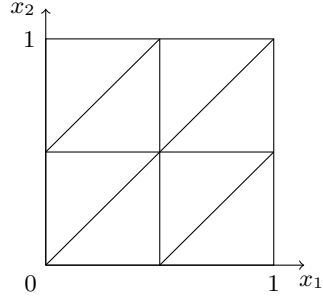

\subsubsection{Using the $L^2(Q)$ projection $\Pi_h = \Pi_h^0$ for the right-hand side $\vec{\mathcal J}$} \label{Sec:FEM:WithoutHT:ProjKonst}

We present the calculation of the right-hand side $\vec{\mathcal J}$ of the linear system~\eqref{FEM:WithoutHT:LGS} when $\Pi_h$ is the $L^2(Q)$ projection onto $S^0(\mathcal T^t_\alpha) \otimes S_d^0(\mathcal T^x_\nu)$ satisfying
\[
    (\Pi_h^0 \vec j, \vec w_h)_{L^2(Q)} = (\vec j, \vec w_h)_{L^2(Q)}
\]
for all $ \vec w_h \in S^0(\mathcal T^t_\alpha) \otimes S_d^0(\mathcal T^x_\nu).$ Then the projection $\Pi_h^0 \vec j $ is the solution of the linear system
\begin{equation} \label{FEM:WithoutHT:jProjKonst}
	M_{h_t}^0 \otimes M_{h_x}^0 (\vec j^1, \dots, \vec j^{N^t})^\top = (\hat{\vec j}^1,\dots, \hat{ \vec j}^{N^t})^\top
\end{equation}
with the matrices and vectors
\begin{align*}
	M_{h_t}^0[l,k] &= (\varphi^0_k,\varphi^0_l)_{L^2(0,T)},  &&l,k=1,\dots,N^t, \\
	M_{h_x}^0[\ell,\kappa] &= (\vec \psi^0_\kappa, \vec \psi^0_\ell)_{L^2(\Omega)},  &&\ell,\kappa=1,\dots,N_x^0, \\
	\vec j^k[\kappa] &= j^k_\kappa ,  &&k=1,\dots,N^t, \; \kappa=1,\dots, N_x^0, \\
	\hat{\vec j}^k[\kappa] &= (\vec j, \varphi^0_k \vec \psi^0_\kappa)_{L^2(Q)} ,  &&k=1,\dots,N^t, \; \kappa=1,\dots, N_x^0,
\end{align*}
through the representation
\begin{equation*}
	\Pi_h^0 \vec j(t,x) = \sum_{k=1}^{N^t} \sum_{\kappa=1}^{N_x^0} j^k_\kappa \varphi^0_k(t) \vec \psi^0_\kappa(x), \quad (t,x) \in Q.
\end{equation*}
Thus, using this representation for the right-hand side \eqref{FEM:WithoutHT:RSf} of the space-time Galerkin--Petrov method \eqref{FEM:WithoutHT:LGS} yields
\begin{equation*}
	f_\ell^l = (\Pi_h^0 \vec j, \varphi^1_l \vec \psi^\mathcal{N}_{\ell} )_{L^2(Q)} =  \sum_{k=1}^{N^t} \sum_{\kappa=1}^{N_x^0} j^k_\kappa \, \underbrace{ (\varphi^0_k, \varphi^1_l)_{L^2(0,T)}}_{= M_{h_t}^{1,0}[l,k]} \, \underbrace{ (\vec \psi^0_\kappa, \vec \psi^\mathcal{N}_{\ell})_{L^2(\Omega)}}_{=M_{h_x}^{\mathcal{N},0}[\ell,\kappa]}  = F[\ell,l]
\end{equation*}
for $\ell=1,\dots,N_x^{\mathcal{N}},$ $l = 0,\dots, N^t-1$ with the matrix
\begin{equation*}
	F = M_{h_x}^{\mathcal{N},0} J ( M_{h_t}^{1,0})^\top  \in \R^{N_x^{\mathcal{N}} \times N^t},
\end{equation*}
where
\begin{align}
	M_{h_x}^{\mathcal{N},0}[\ell,\kappa] &= (\vec \psi^0_\kappa, \vec \psi^\mathcal{N}_{\ell})_{L^2(\Omega)}  &&\ell=1,\dots,N_x^{\mathcal{N}}, \; \kappa=1,\dots,N_x^0, \label{FEM:WithoutHT:MasseN0} \\
	J[\kappa,k] &= j_\kappa^k,  &&\kappa=1,\dots,N_x^0, \; k=1,\dots, N^t, \label{FEM:WithoutHT:JKonst} \\
	M_{h_t}^{1,0}[l,k] &= (\varphi^0_k, \varphi^1_l)_{L^2(0,T)}, &&l=0,\dots,N^t-1, \; k=1,\dots, N^t.  \nonumber
\end{align}

\subsubsection{Using the $L^2(Q)$ projection $\Pi_h=\Pi_h^{\mathcal{RT},1}$ for the right-hand side $\vec{\mathcal J}$} \label{Sec:FEM:WithoutHT:ProjRT}

Analogously to Subsection~\ref{Sec:FEM:WithoutHT:ProjKonst}, we present the calculation of the right-hand side $\vec{\mathcal J}$ of the linear system~\eqref{FEM:WithoutHT:LGS} when $\Pi_h$ is the $L^2(Q)$ projection onto $S^1(\mathcal T^t_\alpha) \otimes \mathcal{RT}^0(\mathcal T^x_\nu)$ given in \eqref{FES:L2RT}. The projection $\Pi_h^{\mathcal{RT},1} \vec j$ is the solution of the linear system
\begin{equation} \label{FEM:WithoutHT:jProjRT}
	\widetilde{M}_{h_t}^1 \otimes M_{h_x}^\mathcal{RT} ( \vec j^0, \vec j^1, \dots, \vec j^{N^t})^\top = (\hat{\vec j}^0,\hat{\vec j}^1,\dots, \hat{ \vec j}^{N^t})^\top
\end{equation}
with the matrices and vectors
\begin{align}
	\widetilde{M}_{h_t}^1[l,k] &= (\varphi^1_k,\varphi^1_l)_{L^2(0,T)},  &&l,k=0,\dots,N^t,  \label{FEM:WithoutHT:MatrixMhtVoll} \\
	M_{h_x}^\mathcal{RT}[\ell,\kappa] &= (\vec \psi^\mathcal{RT}_\kappa, \vec \psi^\mathcal{RT}_\ell)_{L^2(\Omega)},  &&\ell,\kappa=1,\dots,N_x^{\mathcal{RT}}, \nonumber \\
	\vec j^k[\kappa] &= j^k_\kappa ,  &&k=0,\dots,N^t, \; \kappa=1,\dots, N_x^{\mathcal{RT}}, \nonumber \\
	\hat{\vec j}^k[\kappa] &= (\vec j, \varphi^1_k \vec \psi^\mathcal{RT}_\kappa)_{L^2(Q)} ,  &&k=0,\dots,N^t, \; \kappa=1,\dots, N_x^{\mathcal{RT}},  \nonumber
\end{align}
by the representation
\begin{equation*}
	\Pi_h^{\mathcal{RT},1} \vec j(t,x) = \sum_{k=0}^{N^t} \sum_{\kappa=1}^{N_x^{\mathcal{RT}}} j^k_\kappa \varphi^1_k(t) \vec \psi^\mathcal{RT}_\kappa(x), \quad (t,x) \in Q.
\end{equation*}
Thus, using this representation for the right-hand side \eqref{FEM:WithoutHT:RSf} yields
\begin{equation*}
	f_\ell^l = (\Pi_h^{\mathcal{RT},1} \vec j, \varphi^1_l \vec \psi^\mathcal{N}_{\ell} )_{L^2(Q)} =  \sum_{k=0}^{N^t} \sum_{\kappa=1}^{N_x^{\mathcal{RT}}} j^k_\kappa \, \underbrace{ (\varphi^1_k, \varphi^1_l)_{L^2(0,T)}}_{=\widehat M_{h_t}^1[l,k]} \, \underbrace{ (\vec \psi^\mathcal{RT}_\kappa, \vec \psi^\mathcal{N}_{\ell})_{L^2(\Omega)}}_{=M_{h_x}^{\mathcal{N},\mathcal{RT}}[\ell,\kappa]}  = F[\ell,l]
\end{equation*}
for  $\ell=1,\dots,N_x^{\mathcal{N}},$ $l = 0,\dots, N^t-1$ with the matrix
\begin{equation*}
	F = M_{h_x}^{\mathcal{N},\mathcal{RT}} J (\widehat M_{h_t}^1)^\top  \in \R^{N_x^{\mathcal{N}} \times N^t},
\end{equation*}
where
\begin{align}
	M_{h_x}^{\mathcal{N},\mathcal{RT}}[\ell,\kappa] &= (\vec \psi^\mathcal{RT}_\kappa, \vec \psi^\mathcal{N}_{\ell})_{L^2(\Omega)},  &&\ell=1,\dots,N_x^{\mathcal{N}}, \; \kappa=1,\dots,N_x^{\mathcal{RT}}, \label{FEM:WithoutHT:MasseNRT} \\
	J[\kappa,k] &= j_\kappa^k,  &&\kappa=1,\dots,N_x^{\mathcal{RT}}, \; k=0,\dots, N^t, \label{FEM:WithoutHT:Jdiv} \\
	\widehat M_{h_t}^1[l,k] &= (\varphi^1_k, \varphi^1_l)_{L^2(0,T)},  &&l=0,\dots,N^t-1, \; k=0,\dots, N^t.  \label{FEM:WithoutHT:MatrixMhtDach}
\end{align}
Note that the matrices $M_{h_t}^1 \in \R^{N^t \times N^t}$ in \eqref{FEM:WithoutHT:TimeMatrices} and $\widehat M_{h_t}^1 \in \R^{N^t \times (N^t+1)}$ in \eqref{FEM:WithoutHT:MatrixMhtDach} are submatrices of the matrix $\widetilde{M}_{h_t}^1 \in \R^{(N^t + 1) \times (N^t+1)}$ in \eqref{FEM:WithoutHT:MatrixMhtVoll}.

\subsection{Galerkin--Bubnov FEM}\label{Sec:GalerkinBubnovFEM}

In this subsection, we discretize the variational formulation~\eqref{VF:WaveVecHT} where we apply the modified Hilbert transformation. We use a tensor-product ansatz and the conforming finite element space~\eqref{FES:TPSpaceA0}. In particular, we consider the Galerkin--Bubnov finite element method to find $\vec A_h \in S_{0,}^1(\mathcal T^t_\alpha) \otimes \mathcal N_\mathrm{I,0}^0(\mathcal T^x_\nu)$ such that
\begin{equation}  \label{FEM:WithHT:DiskVFMitHT}
	(\epsilon \mathcal H_T \partial_t \vec A_h, \partial_t \vec v_h )_{L^2(Q)} + ( \mu^{-1}\cu_x \vec A_h, \cu_x \mathcal H_T \vec v_h )_{L^2(Q)}  = (\Pi_h \vec j, \mathcal H_T \vec v_h)_{L^2(Q)}
\end{equation}
for all $\vec v_h \in S_{0,}^1(\mathcal T^t_\alpha) \otimes \mathcal N_\mathrm{I,0}^0(\mathcal T^x_\nu)$. Again, $\Pi_h$ is either the $L^2(Q)$ projection $\Pi_h^0$ defined in \eqref{FES:L2Konst} or the $L^2(Q)$ projection $\Pi_h^{\mathcal{RT},1}$ given in \eqref{FES:L2RT}.

The discrete variational formulation~\eqref{FEM:WithHT:DiskVFMitHT} is equivalent to the linear system
\begin{equation} \label{FEM:WithHT:LGS}
	(A_{h_t}^{\mathcal H_T} \otimes M_{h_x}^{\mathcal{N}} +  M_{h_t}^{\mathcal H_T} \otimes A_{h_x}^{\mathcal{N}}) \vec{ \mathcal A }= \vec{\mathcal J}^{\mathcal H_T}
\end{equation}
with the spatial matrices $A_{h_x}^{\mathcal{N}}, M_{h_x}^{\mathcal{N}}$ given in \eqref{FEM:WithoutHT:Ortsmatrizen} and
the temporal matrices
\begin{equation} \label{FEM:WithHT:TimeMatrices}
	A_{h_t}^{\mathcal H_T}[l,k] = (\mathcal H_T \partial_t \varphi^1_k, \partial_t \varphi^1_l)_{L^2(0,T)}, \quad M_{h_t}^{\mathcal H_T}[l,k] = (\varphi^1_k, \mathcal H_T \varphi^1_l)_{L^2(0,T)}
\end{equation}
for $l,k=1,\dots,N^t,$ where $\mathcal H_T$, defined in Section~\ref{Sec:HT}, acts solely on time-dependent functions. As in Subsection~\ref{Sec:FEM:WithoutHT}, we use again the representation~\eqref{FEM:DarstellungAh} of $\vec A_h$ and the vector $\vecA \in \R^{N^t N_x^{\mathcal{N}}}$ of its coefficients~\eqref{FEM:WithoutHT:VectorAh}. The right-hand side of the linear system~\eqref{FEM:WithHT:LGS} is given by
\begin{equation*}
	\vec{\mathcal J}^{\mathcal H_T} = ( \vec{\mathcal J}^1, \vec{\mathcal J}^2, \dots, \vec{\mathcal J}^{N^t} )^\top \in \R^{N^t N_x^{\mathcal{N}}}
\end{equation*}
with
\begin{equation*}
	\vec{\mathcal J}^l = ({\mathcal J}_1^l, {\mathcal J}_2^l, \dots, {\mathcal J}_{N_x^{\mathcal{N}}}^l)^\top \in \R^{N_x^{\mathcal{N}}} \quad \text{ for } l = 1,\dots, N^t,
\end{equation*}
where
\begin{equation} \label{FEM:WithHT:RechteSeite}
	{\mathcal J}_\ell^l = (\Pi_h \vec j, \vec \psi^\mathcal{N}_{\ell} \mathcal H_T \varphi^1_l)_{L^2(Q)} \quad \text{ for } l=1,\dots, N^t, \, \ell=1,\dots,N_x^{\mathcal{N}}.
\end{equation}

Let us take a look at the solvability of the system matrix of the linear system \eqref{FEM:WithHT:LGS}. The temporal matrices $A_{h_t}^{\mathcal H_T}$, $M_{h_t}^{\mathcal H_T}$ in \eqref{FEM:WithHT:TimeMatrices} are positive definite due to property~\eqref{HT:positiv}. In addition, the spatial matrix $M_{h_x}^{\mathcal{N}}$ in \eqref{FEM:WithoutHT:Ortsmatrizen} is also positive definite, whereas the spatial matrix $A_{h_x}^{\mathcal{N}}$ in \eqref{FEM:WithoutHT:Ortsmatrizen} is only positive semi-definite. Hence, the Kronecker product $A_{h_t}^{\mathcal H_T} \otimes M_{h_x}^{\mathcal{N}}$ is positive definite. On the other hand, the product $ M_{h_t}^{\mathcal H_T} \otimes A_{h_x}^{\mathcal{N}}$ is positive semi-definite. Adding both results in a positive definite system matrix of the linear system~\eqref{FEM:WithHT:LGS}. In other words, the linear system~\eqref{FEM:WithHT:LGS} is uniquely solvable. Note that, compared to the static case, we do not need any stabilization to get unique solvability of the linear system~\eqref{FEM:WithHT:LGS} and hence of the corresponding discrete variational formulation~\eqref{FEM:WithHT:DiskVFMitHT}. However, further details on the numerical analysis of the Galerkin--Bubnov finite element method~\eqref{FEM:WithHT:DiskVFMitHT} are far beyond the scope of this contribution, we refer to \cite{LoescherSteinbachZankWelleTheorie, LoescherSteinbachZankDD} for the case of the scalar wave equation.

In addition, note that the temporal matrices $A_{h_t}^{\mathcal H_T}$, $M_{h_t}^{\mathcal H_T}$ in \eqref{FEM:WithHT:TimeMatrices} are dense, whereas the spatial matrices $A_{h_x}^{\mathcal{N}}, M_{h_x}^{\mathcal{N}}$ in \eqref{FEM:WithoutHT:Ortsmatrizen} are sparse. Thus, the linear system~\eqref{FEM:WithHT:LGS} does not allow for a realization as a multistep method. However, the application of fast (direct) solvers, as known for heat and scalar wave equations \cite{LangerZankSISC2021,ZankWelleLoeser,ZankWaermeLoeserICOSAHOM2022}, is possible, which is the topic of future work.

\subsubsection{Using the $L^2(Q)$ projection $\Pi_h = \Pi_h^0$ for the right-hand side $\vec{\mathcal J}^{\mathcal H_T}$} \label{Sec:FEM:WithHT:ProjKonst}

In this subsection, we present the calculation of the right-hand side $\vec{\mathcal J}^{\mathcal H_T}$ of the linear system~\eqref{FEM:WithHT:LGS} when $\Pi_h$ is the $L^2(Q)$ projection onto $S^0(\mathcal T^t_\alpha) \otimes S_d^0(\mathcal T^x_\nu)$ given in \eqref{FES:L2Konst}. Since this subsection is similar to Subsection~\ref{Sec:FEM:WithoutHT:ProjKonst}, we skip the details.

Then the entries~\eqref{FEM:WithHT:RechteSeite} of the right-hand side $\vec{\mathcal J}^{\mathcal H_T}$ admit the representation
\begin{equation*}
	\mathcal J_\ell^l = (\Pi_h^0 \vec j, \vec \psi^\mathcal{N}_{\ell} \mathcal H_T \varphi^1_l)_{L^2(Q)} = F^{\mathcal H_T}[\ell,l]
\end{equation*}
for $\ell=1,\dots,N_x^{\mathcal{N}},$ $l = 1,\dots, N^t$ with the matrix
\begin{equation*}
	F^{\mathcal H_T} = M_{h_x}^{\mathcal{N},0} J (M_{h_t}^{\mathcal H_T1,0})^\top  \in \R^{N_x^{\mathcal{N}} \times N^t},
\end{equation*}
where $M_{h_x}^{\mathcal{N},0}$ is defined in \eqref{FEM:WithoutHT:MasseN0}, $J$ is given in \eqref{FEM:WithoutHT:JKonst} and
\begin{equation*}
	M_{h_t}^{\mathcal H_T1,0}[l,k] = (\varphi^0_k, \mathcal H_T \varphi^1_l)_{L^2(0,T)}, \quad l=1,\dots,N^t, \; k=1,\dots, N^t.
\end{equation*}

\subsubsection{Using the $L^2(Q)$ projection $\Pi_h=\Pi_h^{\mathcal{RT},1}$ for the right-hand side $\vec{\mathcal J}^{\mathcal H_T}$}

Analogously to Subsection~\ref{Sec:FEM:WithHT:ProjKonst}, we present the calculation of the right-hand side $\vec{\mathcal J}^{\mathcal H_T}$ of the linear system~\eqref{FEM:WithHT:LGS} when $\Pi_h$ is the $L^2(Q)$ projection onto $S^1(\mathcal T^t_\alpha) \otimes \mathcal{RT}^0(\mathcal T^x_\nu)$ given in \eqref{FES:L2RT}. As for Subsection~\ref{Sec:FEM:WithHT:ProjKonst}, we skip the details, which are analogous to Subsection~\ref{Sec:FEM:WithoutHT:ProjRT}.

Then the entries~\eqref{FEM:WithHT:RechteSeite}  of the right-hand side $\vec{\mathcal J}^{\mathcal H_T}$ admit the representation
\begin{equation*}
	\mathcal J_\ell^l = (\Pi_h^{\mathcal{RT},1} \vec j, \vec \psi^\mathcal{N}_{\ell} \mathcal H_T \varphi^1_l)_{L^2(Q)} = F^{\mathcal H_T}[\ell,l]
\end{equation*}
for  $\ell=1,\dots,N_x^{\mathcal{N}},$ $l = 1,\dots, N^t$ with the matrix
\begin{equation*}
	F^{\mathcal H_T} = M_{h_x}^{\mathcal{N},\mathcal{RT}} J (\widehat M_{h_t}^{\mathcal H_T})^\top  \in \R^{N_x^{\mathcal{N}} \times N^t},
\end{equation*}
where $M_{h_x}^{\mathcal{N},\mathcal{RT}}$ is defined in \eqref{FEM:WithoutHT:MasseNRT}, $J$ is given in \eqref{FEM:WithoutHT:Jdiv} and
\begin{equation}  \label{FEM:WithHT:MasseHTDach}
	\widehat M_{h_t}^{\mathcal H_T}[l,k] = (\varphi^1_k, \mathcal H_T \varphi^1_l)_{L^2(0,T)}, \quad  l=1,\dots,N^t, \; k=0,\dots, N^t.
\end{equation}
Note that the matrix $M_{h_t}^{\mathcal H_T} \in \R^{N^t \times N^t}$ in \eqref{FEM:WithHT:TimeMatrices} is a submatrix of $\widehat M_{h_t}^{\mathcal H_T} \in \R^{N^t \times (N^t+1)}$ in \eqref{FEM:WithHT:MasseHTDach}.

\section{Numerical examples for the vectorial wave equation} \label{Sec:Num}

In this section, we give numerical examples of the conforming space-time finite element methods~\eqref{FEM:DiskVFOhneHT} and \eqref{FEM:WithHT:DiskVFMitHT} for a spatially two-dimensional domain $\Omega \subset \R^2$, i.e. $d=2$. For this purpose, we consider the unit square $\Omega = (0,1) \times (0,1)$, and set $\epsilon(x) = \begin{pmatrix}
	1 & 0 \\
	0 & 1   
\end{pmatrix}$, 
$\mu(x)=1$ for $x \in \Omega$. The spatial meshes $\mathcal T^x_\nu$ are given by uniform decompositions of the spatial domain $\Omega$ into isosceles right triangles, where a uniform refinement strategy is applied, see Figure~\ref{Num:Fig:NetzOrt}.
\begin{figure}[ht]
	\begin{center}
		\includegraphics[scale=0.6]{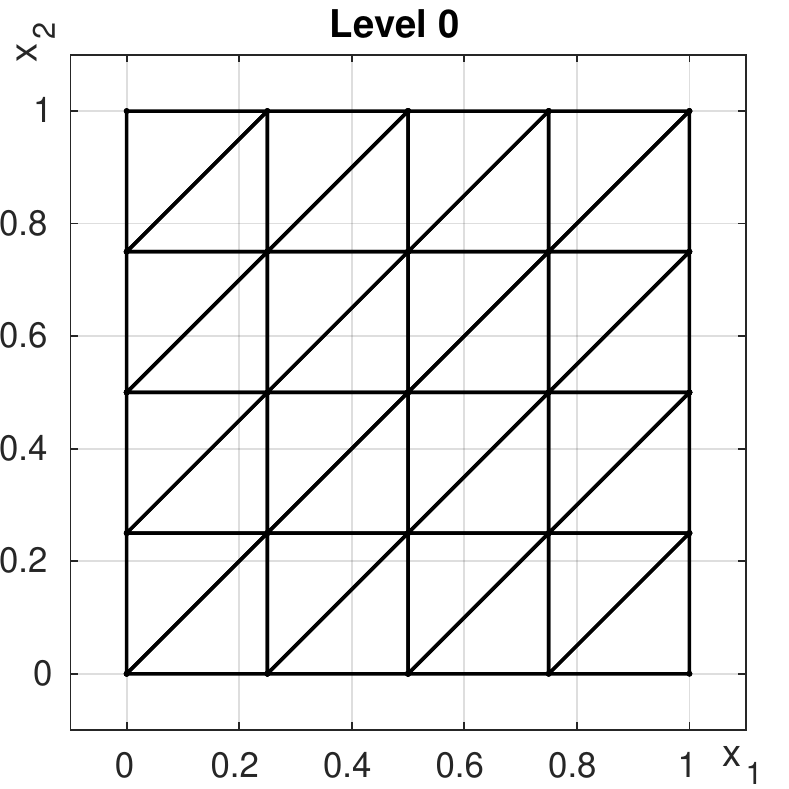}
		\includegraphics[scale=0.6]{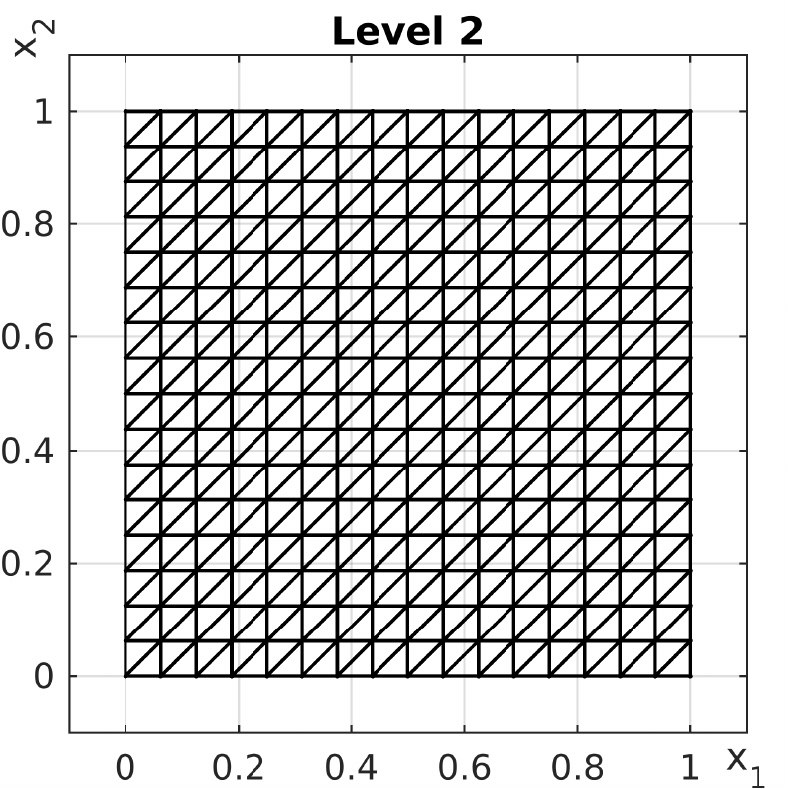}
		\caption{Spatial meshes $\mathcal T^x_\nu$: Initial mesh $\mathcal T^x_0$ and the mesh $\mathcal T^x_2$ after two uniform refinements. }
		\label{Num:Fig:NetzOrt}
	\end{center}
\end{figure}
We investigate the terminal times $T \in \left\{ \sqrt{2}, \frac 3 2 \right\}$ to examine the CFL condition~\eqref{FEM:WithoutHT:CFLQuadrat} of the Galerkin--Petrov finite element method~\eqref{FEM:DiskVFOhneHT}. The temporal meshes $\mathcal T^t_\alpha$ are defined by $t_l = \frac{Tl}{N^t_\alpha}$ for $l=0,\dots,N^t_\alpha,$ where $N^t_\alpha = 5 \cdot 2^\alpha$, $\alpha=0,\dots,4$. Note that for this choice of spatial and temporal meshes, we are in the framework of the CFL condition~\eqref{FEM:WithoutHT:CFLQuadrat}.

In the following numerical examples, we measure the error of the space-time finite element methods~\eqref{FEM:DiskVFOhneHT} and \eqref{FEM:WithHT:DiskVFMitHT} in the space-time norms $\norm{\cdot}_{L^2(Q)}$ and $\abs{\cdot}_{H^{\cu;1}(Q)}$. In particular, we state the numerical results for $\norm{\vec A - \vec A_h}_{L^2(Q)}$ and
\begin{equation*}
	\abs{\vec A - \vec A_h}_{H^{\cu;1}(Q)} = \left( \norm{\partial_t\vec A - \partial_t \vec A_h}_{L^2(Q)}^2 + \norm{\cu_x \vec A - \cu_x \vec A_h}_{L^2(Q)}^2 \right)^{1/2},
\end{equation*}
where $\vec A \in H^{\cu;1}_{0;0,} (Q)$ is the solution of the variational formulation~\eqref{VF:WaveVec}, and $\vec A_h \in S_{0,}^1(\mathcal T^t_\alpha) \otimes \mathcal N_\mathrm{I,0}^0(\mathcal T^x_\nu)$ is the solution of space-time finite element method~\eqref{FEM:DiskVFOhneHT} or \eqref{FEM:WithHT:DiskVFMitHT}. For the vectorial wave equation~\eqref{Einf:Maxwell} and its variational formulation~\eqref{VF:WaveVec}, we use the manufactured solution
\begin{equation}  \label{Num:Lsg}
	\vec A (t,x_1,x_2) = \begin{pmatrix}
		A_1(t,x_1,x_2) \\
		A_2(t,x_1,x_2)
	\end{pmatrix}
	=
	\begin{pmatrix}
		-5 t^2 x_2 (1-x_2) \\
		t^2 x_1 (1-x_1)
	\end{pmatrix}+ t^3
	\begin{pmatrix}
		\sin(\pi x_1) x_2 (1-x_2) \\
		0
	\end{pmatrix}
\end{equation}
for $(t,x_1,x_2) \in \overline{Q},$ which is, for comparision, also investigated in \cite{HauserOhm2023}. This solution satisfies the homogeneous Dirichlet condition $\trt \vec A (t,x_1,x_2) = \vec A_{|\Sigma}(t,x_1,x_2) \times \vec n_x(x_1,x_2) = 0$ for $(t,x_1,x_2) \in \Sigma$ as well as the homogeneous initial conditions
$\vec A(0,x_1,x_2) = \partial_t \vec A(0,x_1,x_2) = 0$ for $(x_1,x_2) \in \Omega.$
The related right-hand side $\vec j$ is given by
\begin{equation*}
	\vec j(t,x_1,x_2) = \begin{pmatrix}
		-10 (t^2-x_2^2+x_2)  \\
		2 (t^2-x_1^2+x_1)
	\end{pmatrix}
	+
	\begin{pmatrix}
		2t^3\sin(\pi x_1) +6t\sin(\pi x_1) x_2 (1-x_2) \\
		\pi t^3(1-2x_2)\cos(\pi x_1)
	\end{pmatrix}, \; (t,x_1,x_2) \in Q.
\end{equation*}

For the computations, we use high-order quadrature rules to calculate the integrals for computing the projections $\Pi_h^0 \vec j$, $\Pi_h^{\mathcal{RT},1} \vec j$ in \eqref{FEM:WithoutHT:jProjKonst}, \eqref{FEM:WithoutHT:jProjRT}. The temporal matrices involving the modified Hilbert transformation $\mathcal H_T$, e.g. the matrices~\eqref{FEM:WithHT:TimeMatrices}, are assembled as proposed in \cite[Subsection~2.2]{ZankCMAM2021}, see also \cite{ZankInt2022} for further assembling strategies. All spatial matrices, e.g. the matrices~\eqref{FEM:WithoutHT:Ortsmatrizen}, are calculated with the help of the finite element library Netgen/NGSolve, see \url{www.ngsolve.org} and \cite{SchoeberlNetgen}. We solve the linear systems using the sparse direct solver  UMFPACK 5.7.1 \cite{Umfpack} in the standard configuration. All calculations, presented in this section, were performed on a PC with two Intel Xeon E5-2687W v4 CPUs 3.00 GHz, i.e. in sum 24 cores and 512 GB main memory.

\begin{table}[h!t]
	\begin{center}
		\begin{footnotesize}\begin{tabular}{r||c|c|c|c|c}
				\diagbox{$h_x$}{\vspace*{-.1cm}$h_t$} & 0.2828 & 0.1414 & 0.0707 & 0.0354 & 0.0177   \\
				\hline\hline
				0.1768 & 6.64e-01 & 6.53e-01 & 6.50e-01 & 6.49e-01 & 6.49e-01 \\
				0.0884 & 3.49e-01 & 3.27e-01 & 3.21e-01 & 3.20e-01 & 3.20e-01 \\
				0.0442 & 2.12e-01 & 1.74e-01 & 1.63e-01 & 1.60e-01 & 1.59e-01 \\
				0.0221 & 1.61e-01 & 1.06e-01 & 8.68e-02 & 8.13e-02 & 7.99e-02 \\
				0.0110 & 1.46e-01 & 8.04e-02 & 5.29e-02 & 4.34e-02 & 4.07e-02
		\end{tabular}\end{footnotesize}
		\caption{Interpolation errors in $\abs{\cdot}_{H^{\cu;1}(Q)}$ for $T=\sqrt{2}$, the unit square $\Omega$ and $\vec A$ in \eqref{Num:Lsg}.} \label{Num:Tab:IntH1Curl}
	\end{center}
\end{table}

\begin{table}[ht!]
	\begin{center}
		\begin{footnotesize}\begin{tabular}{r||c|c|c|c|c}
				\diagbox{$h_x$}{\vspace*{-.1cm}$h_t$} & 0.2828 & 0.1414 & 0.0707 & 0.0354 & 0.0177   \\
				\hline\hline
				0.1768 & 7.50e-02 & 7.49e-02 & 7.49e-02 & 7.49e-02 & 7.49e-02 \\
				0.0884 & 1.99e-02 & 1.93e-02 & 1.93e-02 & 1.93e-02 & 1.93e-02 \\
				0.0442 & 6.97e-03 & 4.96e-03 & 4.82e-03 & 4.81e-03 & 4.81e-03 \\
				0.0221 & 5.25e-03 & 1.74e-03 & 1.24e-03 & 1.20e-03 & 1.20e-03 \\
				0.0110 & 5.13e-03 & 1.31e-03 & 4.37e-04 & 3.11e-04 & 3.01e-04
		\end{tabular}\end{footnotesize}
		\caption{Interpolation errors in $\norm{\cdot}_{L^2(Q)}$ for $T=\sqrt{2}$, the unit square $\Omega$ and $\vec A$ in \eqref{Num:Lsg}.} \label{Num:Tab:IntL2}
	\end{center}
\end{table}

Last, in Table~\ref{Num:Tab:IntH1Curl}, Table~\ref{Num:Tab:IntL2}, we report the interpolation error in the norms $\abs{\cdot}_{H^{\cu;1}(Q)}$ and $\norm{\cdot}_{L^2(Q)}$ for the unit square $\Omega$ and $T=\sqrt{2}$ for the function $\vec A$ defined in \eqref{Num:Lsg}, where first-order convergence is observed in $\abs{\cdot}_{H^{\cu;1}(Q)}$ and second-order convergence is obtained in $\norm{\cdot}_{L^2(Q)}$.

\subsection{Galerkin--Petrov FEM} \label{Sec:Num:Ohne}

In this subsection, we investigate numerical examples for the Galerkin--Petrov finite element method~\eqref{FEM:DiskVFOhneHT} in the situation described at the beginning of this section.

\begin{table}[ht!]
	\begin{center}
		\begin{footnotesize}\begin{tabular}{r||c|c|c|c|c}
				\diagbox{$h_x$}{\vspace*{-.1cm}$h_t$} & 0.2828 & 0.1414 & 0.0707 & 0.0354 & 0.0177   \\
				\hline\hline
				0.1768 & 6.68e-01 & 6.55e-01 & 6.52e-01 & 6.52e-01 & 6.51e-01 \\
				0.0884 & 3.56e-01 & 9.03e-01 & 3.27e-01 & 3.26e-01 & 3.26e-01 \\
				0.0442 & 2.18e-01 & 7.22e-01 & 4.37e+03 & 1.64e-01 & 1.63e-01 \\
				0.0221 & 1.66e-01 & 2.19e-01 & 2.02e+04 & 8.75e+13 & 8.19e-02 \\
				0.0110 & 1.50e-01 & 9.82e-02 & 8.64e+03 & 5.02e+17 & 8.43e+21
		\end{tabular}\end{footnotesize}
		\caption{Errors in $\abs{\cdot}_{H^{\cu;1}(Q)}$ of the Galerkin--Petrov FEM~\eqref{FEM:DiskVFOhneHT} with the approximate right-hand side $\Pi_h^0 \vec j \in S^0(\mathcal T^t_\alpha) \otimes S_2^0(\mathcal T^x_\nu)$ for $T=\sqrt{2}$, the unit square $\Omega$ and $\vec A$ in \eqref{Num:Lsg}.} \label{Num:Tab:OhneConstH1Curl}
	\end{center}
\end{table}

\begin{table}[ht!]
	\begin{center}
		\begin{footnotesize}\begin{tabular}{r||c|c|c|c|c}
				\diagbox{$h_x$}{\vspace*{-.1cm}$h_t$} & 0.2828 & 0.1414 & 0.0707 & 0.0354 & 0.0177   \\
				\hline\hline
				0.1768 & 9.79e-02 & 9.73e-02 & 9.73e-02 & 9.73e-02 & 9.73e-02 \\
				0.0884 & 4.83e-02 & 4.95e-02 & 4.67e-02 & 4.67e-02 & 4.67e-02 \\
				0.0442 & 2.64e-02 & 2.47e-02 & 4.27e+01 & 2.33e-02 & 2.33e-02 \\
				0.0221 & 1.70e-02 & 1.22e-02 & 1.09e+02 & 4.55e+11 & 1.17e-02 \\
				0.0110 & 1.37e-02 & 6.63e-03 & 2.69e+01 & 1.56e+15 & 2.28e+19
		\end{tabular}\end{footnotesize}
		\caption{Errors in $\norm{\cdot}_{L^2(Q)}$ of the Galerkin--Petrov FEM~\eqref{FEM:DiskVFOhneHT} with the approximate right-hand side $\Pi_h^0 \vec j \in S^0(\mathcal T^t_\alpha) \otimes S_2^0(\mathcal T^x_\nu)$ for $T=\sqrt{2}$, the unit square $\Omega$ and $\vec A$ in \eqref{Num:Lsg}.} \label{Num:Tab:OhneConstL2}
	\end{center}
\end{table}

\begin{table}[ht!]
	\begin{center}
		\begin{footnotesize}\begin{tabular}{r||c|c|c|c|c}
				\diagbox{$h_x$}{\vspace*{-.1cm}$h_t$} & 0.2828 & 0.1414 & 0.0707 & 0.0354 & 0.0177   \\
				\hline\hline
				0.1768 & 6.38e-01 & 6.26e-01 & 6.23e-01 & 6.22e-01 & 6.22e-01 \\
				0.0884 & 3.42e-01 & 8.26e-01 & 3.10e-01 & 3.09e-01 & 3.08e-01 \\
				0.0442 & 2.16e-01 & 9.97e-01 & 3.70e+03 & 1.55e-01 & 1.54e-01 \\
				0.0221 & 1.72e-01 & 1.06e+00 & 1.81e+04 & 6.10e+13 & 7.73e-02 \\
				0.0110 & 1.59e-01 & 1.24e+00 & 1.91e+04 & 5.71e+18 & 3.25e+21
		\end{tabular}\end{footnotesize}
		\caption{Errors in $\abs{\cdot}_{H^{\cu;1}(Q)}$ of the Galerkin--Petrov FEM~\eqref{FEM:DiskVFOhneHT} with the approximate right-hand side $\Pi_h^{\mathcal{RT},1} \vec j \in S^1(\mathcal T^t_\alpha) \otimes \mathcal{RT}^0(\mathcal T^x_\nu)$ for $T=\sqrt{2}$, the unit square $\Omega$ and $\vec A$ in \eqref{Num:Lsg}.} \label{Num:Tab:OhneRTH1Curl}
	\end{center}
\end{table}

\begin{table}[ht!]
	\begin{center}
		\begin{footnotesize}\begin{tabular}{r||c|c|c|c|c}
				\diagbox{$h_x$}{\vspace*{-.1cm}$h_t$} & 0.2828 & 0.1414 & 0.0707 & 0.0354 & 0.0177   \\
				\hline\hline
				0.1768 & 4.67e-02 & 4.29e-02 & 4.21e-02 & 4.19e-02 & 4.19e-02 \\
				0.0884 & 1.80e-02 & 1.93e-02 & 1.06e-02 & 1.04e-02 & 1.04e-02 \\
				0.0442 & 1.27e-02 & 1.32e-02 & 3.61e+01 & 2.66e-03 & 2.61e-03 \\
				0.0221 & 1.18e-02 & 1.36e-02 & 1.01e+02 & 3.17e+11 & 6.64e-04 \\
				0.0110 & 1.16e-02 & 1.53e-02 & 8.53e+01 & 1.74e+16 & 8.78e+18
		\end{tabular}\end{footnotesize}
		\caption{Errors in $\norm{\cdot}_{L^2(Q)}$ of the Galerkin--Petrov FEM~\eqref{FEM:DiskVFOhneHT} with the approximate right-hand side $\Pi_h^{\mathcal{RT},1} \vec j \in S^1(\mathcal T^t_\alpha) \otimes \mathcal{RT}^0(\mathcal T^x_\nu)$ for $T=\sqrt{2}$, the unit square $\Omega$ and $\vec A$ in \eqref{Num:Lsg}.} \label{Num:Tab:OhneRTL2}
	\end{center}
\end{table}

First, we consider the terminal time $T=\sqrt{2}$. In Table~\ref{Num:Tab:OhneConstH1Curl} and Table~\ref{Num:Tab:OhneConstL2}, we present the numerical results for the Galerkin--Petrov finite element method~\eqref{FEM:DiskVFOhneHT} with the approximate right-hand side $\Pi_h^0 \vec j \in S^0(\mathcal T^t_\alpha) \otimes S_2^0(\mathcal T^x_\nu)$ of Subsection~\ref{Sec:FEM:WithoutHT:ProjKonst}. In Table~\ref{Num:Tab:OhneRTH1Curl} and Table~\ref{Num:Tab:OhneRTL2}, we report the results for $\Pi_h^{\mathcal{RT},1} \vec j  \in S^1(\mathcal T^t_\alpha) \otimes \mathcal{RT}^0(\mathcal T^x_\nu)$ of Subsection~\ref{Sec:FEM:WithoutHT:ProjRT}. All tables show conditional stability, i.e. the CFL condition~\eqref{FEM:WithoutHT:CFLQuadrat} is required for stability. Note that the ratio of the mesh sizes $h_t=0.0177$ and $h_x=0.0221$, i.e. the last column and second last row of Tables~\ref{Num:Tab:OhneConstH1Curl} to \ref{Num:Tab:OhneRTL2}, is given by
\begin{equation*}
	\frac{h_t}{h_x} \approx \frac{0.0177}{0.0221} \approx 0.801
\end{equation*}
and thus, fulfills the CFL condition~\eqref{FEM:WithoutHT:CFLQuadrat} resulting in a stable method. In the case of stability, we observe first-order convergence in $\abs{\cdot}_{H^{\cu;1}(Q)}$, where the errors in Table~\ref{Num:Tab:IntH1Curl}, Table~\ref{Num:Tab:OhneConstH1Curl} and Table~\ref{Num:Tab:OhneRTH1Curl} are within the same range. When considering the errors in $\norm{\cdot}_{L^2(Q)}$, Table~\ref{Num:Tab:OhneConstL2} reports only first-order convergence, whereas in Table~\ref{Num:Tab:OhneRTL2}, second-order convergence is observed as in  Table~\ref{Num:Tab:IntL2} on the interpolation error.

Next, we elaborate on this convergence behavior by investigating the difference $|\vec A(T,x)- \vec A_h(T,x)|$ for $x \in \Omega$. For the mesh sizes $h_t =  0.0707 $, $h_x =  0.0884$, we plot the function $\Omega \ni x \mapsto |\vec A(T,x)- \vec A_h(T,x)|$ in Figure~\ref{Num:Fig:SpuriousModes} for both projections $\Pi_h^0 $ and $\Pi_h^{\mathcal{RT},1}.$ In the literature, we find the term \textit{spurious solutions}, e.g. in \cite{Jin}, or \textit{spurious modes}, e.g. in \cite{bossavit1990}. Spurious modes are parts of the numerical solution, which correspond to eigensolutions of the discrete differential operator. They are oscillations that should not be part of the computed solution. Spurious modes can be understood as numerically generated noise and have no physical meaning. An example can be found in  \cite[Figure~5.8]{Jin}. In Figure~\ref{Num:Fig:SpuriousModes}, we obtain a similar noise behavior that occurs for the projection $\Pi_h^0 \vec j$ of the right-hand side. Here, we observe some of the zero eigensolutions of the curl-curl operator added to the solution. Hence, we get a worse $L^2(Q)$-error, see Table~\ref{Num:Tab:OhneConstL2} and Table~\ref{Num:Tab:OhneRTL2}, but a similar error behavior in the $\abs{\cdot}_{H^{\cu;1}(Q)}$-norm for both projections if the CFL condition is met, see Table~\ref{Num:Tab:OhneConstH1Curl} and Table~\ref{Num:Tab:OhneRTH1Curl}.

\begin{figure}[ht]
	\begin{center}
		\includegraphics[scale=0.27]{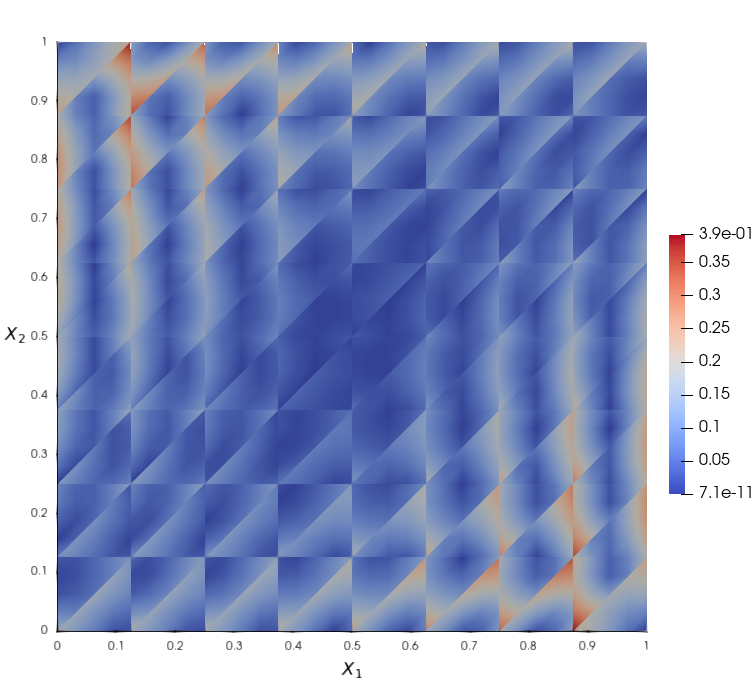}\hspace*{.2cm}
		\includegraphics[scale=0.27]{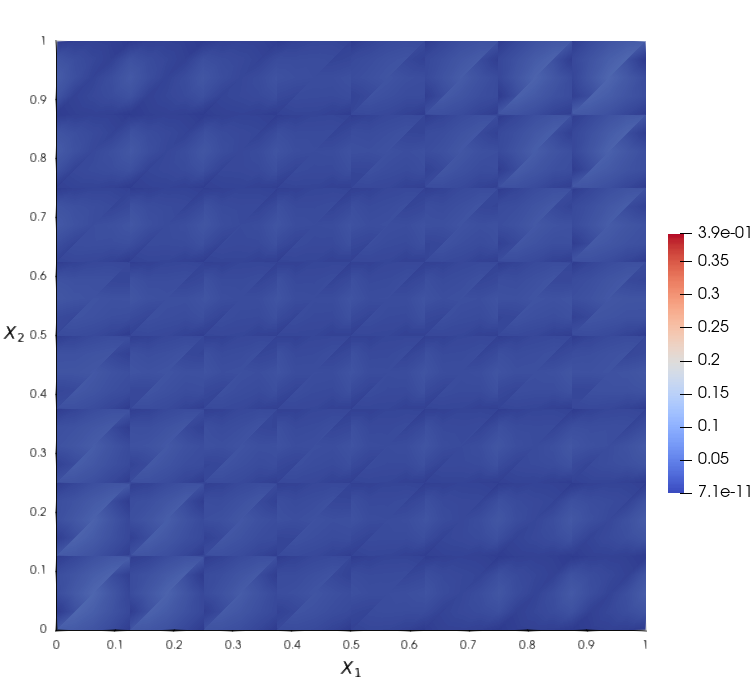}
		\caption{The magnitude of the difference $\abs{\vec A(T,\cdot)- \vec A_h(T,\cdot)}$ for $T=\sqrt{2}$, $\nu = 1$, $\alpha=2$, displayed over $\Omega$ for $\Pi_h^0 \vec j \in S^0(\mathcal T^t_\alpha) \otimes S_2^0(\mathcal T^x_\nu)$ (left) and $\Pi_h^{\mathcal{RT},1} \vec j \in S^1(\mathcal T^t_\alpha) \otimes \mathcal{RT}^0(\mathcal T^x_\nu)$ (right).}
		\label{Num:Fig:SpuriousModes}
	\end{center}
\end{figure}

To summarize, the approximation $\Pi_h^0 \vec j \in S^0(\mathcal T^t_\alpha) \otimes S_2^0(\mathcal T^x_\nu)$ of $\vec j$ is good enough, when the error is measured in $\abs{\cdot}_{H^{\cu;1}(Q)}$, while the better approximation $\Pi_h^{\mathcal{RT},1} \vec j  \in S^1(\mathcal T^t_\alpha) \otimes \mathcal{RT}^0(\mathcal T^x_\nu)$ of $\vec j$ is needed for optimal convergence rates in $\norm{\cdot}_{L^2(Q)}.$ 

Second, we examine the terminal time $T=\frac 3 2$. In Table~\ref{Num:Tab:OhneT32RTH1Curl}, Table~\ref{Num:Tab:OhneT32RTL2}, the numerical results for the Galerkin--Petrov finite element method~\eqref{FEM:DiskVFOhneHT} with the approximate right-hand side $\Pi_h^{\mathcal{RT},1} \vec j  \in S^1(\mathcal T^t_\alpha) \otimes \mathcal{RT}^0(\mathcal T^x_\nu)$ of Subsection~\ref{Sec:FEM:WithoutHT:ProjRT} show that slightly violating the CFL condition~\eqref{FEM:WithoutHT:CFLQuadrat} leads to instability. More precisely, the ratio of the mesh sizes $h_t=0.0188$ and $h_x=0.0221$, i.e. the last column and second last row of Table~\ref{Num:Tab:OhneT32RTH1Curl}, Table~\ref{Num:Tab:OhneT32RTL2}, is given by
\begin{equation*}
	\frac{h_t}{h_x} \approx \frac{0.0188}{0.0221} \approx 0.851
\end{equation*}
and thus, violates the CFL condition~\eqref{FEM:WithoutHT:CFLQuadrat} resulting in an unstable method. In other words, the CFL condition~\eqref{FEM:WithoutHT:CFLQuadrat} seems to be sharp for this particular situation.

\begin{table}[ht!]
	\begin{center}
		\begin{footnotesize}\begin{tabular}{r||c|c|c|c|c}
				\diagbox{$h_x$}{\vspace*{-.1cm}$h_t$} & 0.3000 & 0.1500 & 0.0750 & 0.0375 & 0.0188   \\
				\hline\hline
				0.1768 & 7.31e-01 & 7.18e-01 & 7.16e-01 & 7.15e-01 & 7.15e-01 \\
				0.0884 & 3.90e-01 & 1.09e+00 & 3.57e-01 & 3.55e-01 & 3.55e-01 \\
				0.0442 & 2.44e-01 & 1.24e+00 & 6.37e+03 & 4.40e-01 & 1.77e-01 \\
				0.0221 & 1.92e-01 & 1.37e+00 & 2.28e+04 & 9.80e+14 & 1.29e+04 \\
				0.0110 & 1.77e-01 & 1.58e+00 & 2.44e+04 & 5.36e+17 & 7.84e+21
		\end{tabular}\end{footnotesize}
		\caption{Errors in $\abs{\cdot}_{H^{\cu;1}(Q)}$ of the Galerkin--Petrov FEM~\eqref{FEM:DiskVFOhneHT} with the approximate right-hand side $\Pi_h^{\mathcal{RT},1} \vec j \in S^1(\mathcal T^t_\alpha) \otimes \mathcal{RT}^0(\mathcal T^x_\nu)$ for $T=\frac 3 2$, the unit square $\Omega$ and $\vec A$ in \eqref{Num:Lsg}.} \label{Num:Tab:OhneT32RTH1Curl}
	\end{center}
\end{table}

\begin{table}[ht!]
	\begin{center}
		\begin{footnotesize}\begin{tabular}{r||c|c|c|c|c}
				\diagbox{$h_x$}{\vspace*{-.1cm}$h_t$} & 0.3000 & 0.1500 & 0.0750 & 0.0375 & 0.0188  \\
				\hline\hline
				0.1768 & 5.45e-02 & 5.06e-02 & 4.99e-02 & 4.97e-02 & 4.97e-02 \\
				0.0884 & 2.05e-02 & 2.49e-02 & 1.26e-02 & 1.24e-02 & 1.23e-02 \\
				0.0442 & 1.43e-02 & 1.69e-02 & 6.28e+01 & 4.40e-03 & 3.09e-03 \\
				0.0221 & 1.33e-02 & 1.85e-02 & 1.30e+02 & 5.17e+12 & 4.92e+01 \\
				0.0110 & 1.31e-02 & 2.05e-02 & 1.16e+02 & 1.67e+15 & 2.16e+19
		\end{tabular}\end{footnotesize}
		\caption{Errors in $\norm{\cdot}_{L^2(Q)}$ of the Galerkin--Petrov FEM~\eqref{FEM:DiskVFOhneHT} with the approximate right-hand side $\Pi_h^{\mathcal{RT},1} \vec j \in S^1(\mathcal T^t_\alpha) \otimes \mathcal{RT}^0(\mathcal T^x_\nu)$ for $T=\frac 3 2$, the unit square $\Omega$ and $\vec A$ in \eqref{Num:Lsg}.} \label{Num:Tab:OhneT32RTL2}
	\end{center}
\end{table}

\subsection{Galerkin--Bubnov FEM}

In this subsection, we report on numerical results for the Galerkin--Bubnov finite element method~\eqref{FEM:WithHT:DiskVFMitHT} using the modified Hilbert transformation in the situation described at the beginning of this section. We only show numerical results for the terminal time $T = \sqrt{2}$, as those for $T=\frac 3 2$ are similar.

\begin{table}[ht!]
	\begin{center}
		\begin{footnotesize}\begin{tabular}{r||c|c|c|c|c}
				\diagbox{$h_x$}{\vspace*{-.1cm}$h_t$} & 0.2828 & 0.1414 & 0.0707 & 0.0354 & 0.0177   \\
				\hline\hline
				0.1768 & 6.77e-01 & 6.58e-01 & 6.53e-01 & 6.52e-01 & 6.51e-01 \\
				0.0884 & 3.72e-01 & 3.36e-01 & 3.28e-01 & 3.26e-01 & 3.26e-01 \\
				0.0442 & 2.41e-01 & 1.83e-01 & 1.68e-01 & 1.64e-01 & 1.63e-01 \\
				0.0221 & 1.95e-01 & 1.16e-01 & 9.13e-02 & 8.40e-02 & 8.21e-02 \\
				0.0110 & 1.81e-01 & 9.25e-02 & 5.79e-02 & 4.56e-02 & 4.20e-02
		\end{tabular}\end{footnotesize}
		\caption{Errors in $\abs{\cdot}_{H^{\cu;1}(Q)}$ of the Galerkin--Bubnov FEM~\eqref{FEM:WithHT:DiskVFMitHT} with the approximate right-hand side $\Pi_h^0 \vec j \in S^0(\mathcal T^t_\alpha) \otimes S_2^0(\mathcal T^x_\nu)$ for $T=\sqrt{2}$, the unit square $\Omega$ and $\vec A$ in \eqref{Num:Lsg}.} \label{Num:Tab:MitConstH1Curl}
	\end{center}
\end{table}

\begin{table}[ht!]
	\begin{center}
		\begin{footnotesize}\begin{tabular}{r||c|c|c|c|c}
				\diagbox{$h_x$}{\vspace*{-.1cm}$h_t$} & 0.2828 & 0.1414 & 0.0707 & 0.0354 & 0.0177   \\
				\hline\hline
				0.1768 & 1.03e-01 & 9.61e-02 & 9.70e-02 & 9.73e-02 & 9.73e-02 \\
				0.0884 & 5.20e-02 & 4.61e-02 & 4.65e-02 & 4.67e-02 & 4.67e-02 \\
				0.0442 & 2.98e-02 & 2.33e-02 & 2.32e-02 & 2.32e-02 & 2.33e-02 \\
				0.0221 & 2.08e-02 & 1.22e-02 & 1.16e-02 & 1.16e-02 & 1.16e-02 \\
				0.0110 & 1.79e-02 & 7.12e-03 & 5.91e-03 & 5.84e-03 & 5.83e-03
		\end{tabular}\end{footnotesize}
		\caption{Errors in $\norm{\cdot}_{L^2(Q)}$ of the Galerkin--Bubnov FEM~\eqref{FEM:WithHT:DiskVFMitHT} with the approximate right-hand side $\Pi_h^0 \vec j \in S^0(\mathcal T^t_\alpha) \otimes S_2^0(\mathcal T^x_\nu)$ for $T=\sqrt{2}$, the unit square $\Omega$ and $\vec A$ in \eqref{Num:Lsg}.} \label{Num:Tab:MitConstL2}
	\end{center}
\end{table}

\begin{table}[ht!]
	\begin{center}
		\begin{footnotesize}\begin{tabular}{r||c|c|c|c|c}
				\diagbox{$h_x$}{\vspace*{-.1cm}$h_t$} & 0.2828 & 0.1414 & 0.0707 & 0.0354 & 0.0177   \\
				\hline\hline
				0.1768 & 6.45e-01 & 6.27e-01 & 6.23e-01 & 6.23e-01 & 6.22e-01 \\
				0.0884 & 3.62e-01 & 3.19e-01 & 3.11e-01 & 3.09e-01 & 3.08e-01 \\
				0.0442 & 2.48e-01 & 1.75e-01 & 1.59e-01 & 1.55e-01 & 1.54e-01 \\
				0.0221 & 2.10e-01 & 1.13e-01 & 8.72e-02 & 7.95e-02 & 7.75e-02 \\
				0.0110 & 2.00e-01 & 9.14e-02 & 5.63e-02 & 4.36e-02 & 3.98e-02
		\end{tabular}\end{footnotesize}
		\caption{Errors in $\abs{\cdot}_{H^{\cu;1}(Q)}$ of the Galerkin--Bubnov FEM~\eqref{FEM:WithHT:DiskVFMitHT} with the approximate right-hand side $\Pi_h^{\mathcal{RT},1} \vec j \in S^1(\mathcal T^t_\alpha) \otimes \mathcal{RT}^0(\mathcal T^x_\nu)$ for $T=\sqrt{2}$, the unit square $\Omega$ and $\vec A$ in \eqref{Num:Lsg}.} \label{Num:Tab:MitRTH1Curl}
	\end{center}
\end{table}

\begin{table}[ht!]
	\begin{center}
		\begin{footnotesize}\begin{tabular}{r||c|c|c|c|c}
				\diagbox{$h_x$}{\vspace*{-.1cm}$h_t$} & 0.2828 & 0.1414 & 0.0707 & 0.0354 & 0.0177   \\
				\hline\hline
				0.1768 & 5.28e-02 & 4.23e-02 & 4.20e-02 & 4.19e-02 & 4.19e-02 \\
				0.0884 & 2.70e-02 & 1.10e-02 & 1.05e-02 & 1.04e-02 & 1.04e-02 \\
				0.0442 & 2.28e-02 & 3.99e-03 & 2.75e-03 & 2.61e-03 & 2.59e-03 \\
				0.0221 & 2.21e-02 & 2.91e-03 & 9.92e-04 & 6.85e-04 & 6.52e-04 \\
				0.0110 & 2.19e-02 & 2.79e-03 & 7.22e-04 & 2.46e-04 & 1.71e-04
		\end{tabular}\end{footnotesize}
		\caption{Errors in $\norm{\cdot}_{L^2(Q)}$ of the Galerkin--Bubnov FEM~\eqref{FEM:WithHT:DiskVFMitHT} with the approximate right-hand side $\Pi_h^{\mathcal{RT},1} \vec j \in S^1(\mathcal T^t_\alpha) \otimes \mathcal{RT}^0(\mathcal T^x_\nu)$ for $T=\sqrt{2}$, the unit square $\Omega$ and $\vec A$ in \eqref{Num:Lsg}.} \label{Num:Tab:MitRTL2}
	\end{center}
\end{table}

We observe unconditional stability in Table~\ref{Num:Tab:MitConstH1Curl}, Table~\ref{Num:Tab:MitConstL2}, Table~\ref{Num:Tab:MitRTH1Curl}, Table~\ref{Num:Tab:MitRTL2}, i.e. no CFL condition is needed. This is the main difference between the results in this subsection and the results of Subsection~\ref{Sec:Num:Ohne}, where for stability, a CFL condition is required.

Besides the stability issue, the errors in Table~\ref{Num:Tab:MitConstH1Curl}, Table~\ref{Num:Tab:MitConstL2}, Table~\ref{Num:Tab:MitRTH1Curl}, Table~\ref{Num:Tab:MitRTL2} are comparable with the errors of the previous Subsection~\ref{Sec:Num:Ohne} regarding the projection of the right-hand side. Thus, the approximation $\Pi_h^0 \vec j \in S^0(\mathcal T^t_\alpha) \otimes S_2^0(\mathcal T^x_\nu)$ of $\vec j$ is good enough to get first-order convergence in $\abs{\cdot}_{H^{\cu;1}(Q)}$, see Table~\ref{Num:Tab:MitConstH1Curl}. In Table~\ref{Num:Tab:MitRTL2}, we  again see that the better approximation $\Pi_h^{\mathcal{RT},1} \vec j \in S^1(\mathcal T^t_\alpha) \otimes \mathcal{RT}^0(\mathcal T^x_\nu)$ of $\vec j$ is needed for optimal convergence rates in $\norm{\cdot}_{L^2(Q)}.$

\section{Conclusion} \label{Sec:Zum}

In this paper, we presented two conforming space-time finite element methods for the vectorial wave equation. First, we stated a variational formulation with different trial and test spaces. Its conforming discretization with piecewise multilinear functions leads to a Galerkin--Petrov method, which is only conditionally stable, i.e. a CFL condition is required. For a particular choice of spatial meshes, we stated the CFL condition, where numerical examples showed its sharpness. Second, to tackle the problem of a CFL condition, we introduced a variational formulation for the vectorial wave equation with equal trial and test spaces using the modified Hilbert transformation $\mathcal H_T$. We gave a rigorous derivation of this variational setting. A conforming discretization with piecewise multilinear functions of this new variational approach results in a space-time Galerkin--Bubnov method. All numerical examples showed the unconditional stability of this conforming space-time method. 
Further, we investigated the influence of projections of the right-hand side on the convergence. In numerical examples for both the Galerkin--Petrov method and the Galerkin--Bubnov method, we observed only first-order convergence in $\norm{\cdot}_{L^2(Q)}$ when the right-hand side was projected onto piecewise constants.  On the other side, second-order convergence was obtained when the right-hand side was projected onto piecewise linear functions in time and lowest-order Raviart--Thomas functions in space.

For both presented conforming space-time approaches, a generalization to piecewise polynomials of higher-order is possible, see \cite{HauserOhm2023}. Further, the fast solvers \cite{LangerZankSISC2021, ZankWelleLoeser, ZankWaermeLoeserICOSAHOM2022} for the resulting linear systems, where some allow for time parallelization, are also applicable. In addition, other time-dependent problems in electromagnetics can be handled with the presented variational frameworks, which are topics of future work, see e.g. \cite{HauserOhm2023}.

\bibliographystyle{acm}
\bibliography{references}
\end{document}